\documentclass[11pt, reqno]{amsart}
\usepackage[margin=1.05in]{geometry}
\usepackage{amsmath,amssymb,mathtools}
\usepackage{enumitem}
\usepackage{hyperref}
\usepackage[T1]{fontenc}
\usepackage{lmodern}
\usepackage{xcolor}
\usepackage[normalem]{ulem}

\newcounter{comment}

\newcommand{\com}[2]{%
    \refstepcounter{comment}%
    \ifthenelse{\equal{#1}{r}}%
        {\textcolor{red}{\textbf{[\thecomment]} #2}}%
        {\textcolor{purple}{\textbf{[\thecomment]} #2}}%
}

\usepackage{eucal}
\usepackage{soul}
\usepackage{tikz-cd}
\usetikzlibrary{arrows.meta,positioning}
\usepackage{float}
\sethlcolor{yellow}

\usepackage[
  backend=biber,
  style=numeric,
  maxbibnames=6,
  minbibnames=3
]{biblatex}

\AtBeginBibliography{\small}
\usepackage{cleveref}

\theoremstyle{plain}

\newtheorem{theorem}{Theorem}[section]

\newtheorem{lemma}[theorem]{Lemma}
\newtheorem{corollary}[theorem]{Corollary}

\theoremstyle{definition}

\newtheorem{definition}[theorem]{Definition}
\newtheorem{remark}{Remark}

\theoremstyle{plain}

\newtheorem{question}{Question}

\crefname{question}{question}{questions}
\Crefname{question}{Question}{Questions}

\newcommand{\Ric}{\operatorname{Ric}}
\newcommand{\vol}{\operatorname{vol}}
\newcommand{\diam}{\operatorname{diam}}
\renewcommand{\sec}{\operatorname{sec}}

\newcommand{\abs}[1]{\lvert #1 \rvert}

\newcommand{\Fix}{\operatorname{Fix}}

\makeatletter

\def\l@section{%
  \@tocline{1}{3pt plus 1pt}{0pt}{}{\bfseries}}

\def\l@subsection{%
  \@tocline{2}{0pt}{1.8em}{}{}}

\makeatother

\begin{document}

\title[On the topological stability of Kähler--Einstein 4-Manifolds]
{On the topological stability of Kähler--Einstein 4-Manifolds}

\author[Alejandro Bellati]{Alejandro Bellati$^{1}$}
\author[Ignacio Bustamante]{Ignacio Bustamante$^{2}$}

\makeatletter
\renewcommand{\@setauthors}{%
  \begingroup
  \def\thanks{\protect\thanks@warning}%
  \trivlist
  \centering\footnotesize
  \@topsep30\p@\relax
  \advance\@topsep by -\baselineskip
  \item\relax
  \author@andify\authors
  \def\\{\protect\linebreak}%
  \MakeUppercase{\authors}%
  \endtrivlist
  \endgroup
  \vspace{0.2em}
  {\centering\footnotesize
  $^{1}$Department of Mathematics, Princeton University, USA.
    {\itshape Email address:}
  {\ttfamily abellati@math.princeton.edu}\par
  $^{2}$IMERL, Universidad de la República, Uruguay.
    {\itshape Email address:}
  {\ttfamily ibustamante@fing.edu.uy}\par
  }
  \vspace{0.7em}
}
\makeatother

\date{}

\begin{abstract}
We establish both topological stability and instability results under noncollapsed Gromov--Hausdorff convergence of closed K\"ahler--Einstein $4$-manifolds. 
Regarding instability, we construct sequences of Ricci-flat K\"ahler metrics on $K3$ surfaces and Enriques surfaces that converge to the same compact flat orbifold. We also construct two sequences of closed locally hyperk\"ahler Ricci-flat $4$-manifolds, with fundamental groups $\mathbb Z_2$ and $\mathbb Z_2\times\mathbb Z_2$, respectively, which converge to the same compact flat orbifold. Taken together, these examples realize all three fundamental-group types in Hitchin's classification of closed nonflat locally hyperk\"ahler $4$-manifolds.

In contrast, we prove that the positive index $b_2^+$ is stable under noncollapsed Gromov--Hausdorff convergence for closed locally K\"ahler $4$-manifolds with two-sided Ricci bounds provided that the orientations are compatible. As a consequence, sequences with distinct
fundamental-group types in Hitchin's classification can share a common orbifold limit only if the limit is flat, and no common limit can be shared by all three types.
\end{abstract}

\maketitle 

\enlargethispage{2\baselineskip}
\begingroup
\linespread{0.95}\selectfont
\setlength{\parskip}{0pt}
\tableofcontents
\endgroup

\clearpage

\section{Introduction}

A classical problem in Riemannian geometry is to understand how geometric bounds constrain the topology of a manifold. A typical setting is that of closed $n$-dimensional Riemannian manifolds satisfying a curvature bound together with
\[
\operatorname{diam}(M)\le D,
\qquad
\operatorname{vol}(M)\ge \mathrm{v}>0.
\]
The strength of the resulting topological conclusions depends strongly on the type of curvature control imposed.

A basic reference point is Cheeger's finiteness theorem \cite{Cheeger1970, Cheeger1967} (see also \cite{Peters1984}): under a two-sided sectional curvature bound
\[
|\sec_M|\le C,
\]
together with the diameter and volume bounds above, only finitely many diffeomorphism types can occur.  If the upper sectional curvature bound is removed and one assumes only
\[
\sec_M\ge K,
\]
the limiting theory becomes singular: the class is still precompact in the Gromov--Hausdorff topology, but its closure contains Alexandrov spaces \cite{BuragoGromovPerelman1992}. Nevertheless, a strong form of topological stability survives. By Perelman's stability theorem \cite{Perelman1991,Kapovitch2007}, if
\[
M_i\xrightarrow{GH}X
\]
is a non-collapsing sequence with a uniform lower sectional curvature bound, then, for all sufficiently large $i$, the spaces $M_i$ are homeomorphic to $X$. 

The situation changes markedly when one further replaces sectional curvature control by a lower Ricci curvature bound. Although Gromov--Hausdorff precompactness still holds under the corresponding curvature and diameter bounds, the topological conclusions are substantially weaker. In particular, one can no longer expect control of the homotopy type. Perelman \cite{Perelman1997} constructed manifolds with positive Ricci curvature, uniformly bounded diameter, and volume uniformly bounded away from zero, with arbitrarily large topology.

Despite this loss of control over the homotopy type, lower Ricci curvature still imposes strong restrictions on the fundamental group. A theorem of Anderson \cite{Anderson1990Short} shows that, for fixed $n$, $K$, $D$, and $\mathrm v>0$, among closed $n$-manifolds satisfying
\begin{equation}\label{eq:geometric-bounds}
\operatorname{Ric}_M\ge (n-1)K,
\qquad
\operatorname{diam}(M)\le D,
\qquad
\operatorname{vol}(M)\ge \mathrm{v},
\end{equation}
only finitely many isomorphism types of fundamental groups can occur. When $K\ge0$, further restrictions are available. If $K=0$, the Cheeger--Gromoll splitting theorem implies that the fundamental group is virtually abelian, while for $K>0$ the Bonnet--Myers theorem implies that it is finite.

Anderson's finiteness theorem controls the collection of fundamental groups that may occur under \eqref{eq:geometric-bounds}, but it does not imply that the fundamental group is stable under Gromov--Hausdorff convergence. An example of Otsu \cite{Otsu1991} illustrates this failure: there is a sequence of metrics $g_i$ on $S^3\times\mathbb{RP}^2$ satisfying bounds of the form \eqref{eq:geometric-bounds} such that
\[
(S^3\times\mathbb{RP}^2,g_i)\xrightarrow{GH}\Sigma(S^2\times\mathbb{RP}^2).
\]
Here $\Sigma$ denotes the spherical suspension. The limit is simply connected, whereas
\[
\pi_1(S^3\times\mathbb{RP}^2)\cong\mathbb Z_2.
\]

Some control over the fundamental group survives the limiting procedure. Sormani--Wei \cite{SormaniWei2001} proved that, under a uniform lower Ricci curvature bound and a uniform diameter bound, the Gromov--Hausdorff limit $X$ admits a universal cover and that, for all sufficiently large $i$, there is a surjective homomorphism from $\pi_1(M_i)$ onto its group of deck transformations. Under the additional non-collapsing assumption, Pan--Wei \cite{PanWei2022} showed that this group can be identified with the fundamental group of $X$. Consequently, for all sufficiently large $i$, there is a surjective homomorphism
\[
\pi_1(M_i)\longrightarrow\pi_1(X).
\]

In low dimensions, stronger topological conclusions are available under the bounds \eqref{eq:geometric-bounds}. In dimension two, Perelman's stability theorem applies. In dimension three, Simon \cite[Proposition~9.3]{Simon2012} proved that, along a non-collapsing Gromov--Hausdorff convergent sequence with a uniform lower Ricci curvature bound and a uniform upper diameter bound, the limit is homeomorphic to a three-manifold and the approximating manifolds are mutually diffeomorphic for all sufficiently large indices. See also \cite{BruePigatiSemola2024,DengPigati2025} for more recent topological stability results in the noncollapsed synthetic lower Ricci curvature setting. Thus, dimension four is the first dimension in which fundamental-group instability between smooth manifolds satisfying \eqref{eq:geometric-bounds} may occur.

One could therefore ask whether some stability statement  holds under a lower Ricci curvature bound and a non-collapsing assumption: although only finitely many isomorphism types of fundamental groups can occur under \eqref{eq:geometric-bounds}, must sufficiently close manifolds in this class have the same fundamental group? J. Pan \cite{Pan2025} formulated this question as follows.

\begin{question}\label{q:pan-stability}
For fixed $n,D,\mathrm v>0$, does there exist $\varepsilon=\varepsilon(n,D,\mathrm v)>0$ such that any two closed $n$-manifolds $M$ and $N$ satisfying
\[
\operatorname{Ric}\ge -(n-1),
\qquad
\operatorname{diam}\le D,
\qquad
\operatorname{vol}\ge \mathrm v,
\]
and
\[
d_{GH}(M,N)\le\varepsilon
\]
have isomorphic fundamental groups?
\end{question}

By Gromov's compactness theorem, a negative answer to \Cref{q:pan-stability} is equivalent to the existence of two sequences $(M_i)$ and $(N_i)$ satisfying the same uniform geometric bounds \eqref{eq:geometric-bounds} and such that, after passing to a subsequence,
\[
M_i\xrightarrow{GH}X,
\qquad
N_i\xrightarrow{GH}X,
\qquad
\pi_1(M_i)\not\cong\pi_1(N_i)
\]
for every $i$.

Brena \cite{Brena2026} answered \Cref{q:pan-stability} in the negative by constructing two sequences of closed $4$-manifolds $(M_i)$ and $(N_i)$ satisfying \eqref{eq:geometric-bounds} with $K=0$, and
\[
\pi_1(M_i)\cong\mathbb Z_2,
\qquad
\pi_1(N_i)=1,
\]
such that both sequences converge in the Gromov--Hausdorff sense to the same limit
\[
X=\Sigma(cS^3/\mu_4).
\]
Here $c>0$ is a sufficiently small constant, $\Sigma$ denotes the spherical suspension, and $\mu_4$ acts diagonally on $S^3\subset\mathbb C^2$ by
\[
(z_1,z_2)\longmapsto(iz_1,iz_2).
\]
The quotient $S^3/\mu_4$ is the lens space $L(4,1)$.

In view of this counterexample, we ask whether fundamental-group instability persists under substantially stronger curvature assumptions. Dimension four is particularly interesting in this regard: by the work of Cheeger--Naber \cite{CheegerNaber2015}, uniform two-sided Ricci curvature bounds, together with bounded diameter and non-collapsing, imply diffeomorphism finiteness, recovering the conclusion of Cheeger's finiteness theorem under weaker curvature assumptions. In arbitrary dimension, Anderson--Cheeger \cite{AndersonCheeger1991} established diffeomorphism finiteness under two-sided Ricci curvature bounds and an $L^{n/2}$ bound on the full curvature tensor; see also \cite{JiangWei2025} for extensions under weaker curvature hypotheses.

Here we consider an even more restrictive setting, namely that of Einstein $4$-manifolds, for which additional topological restrictions arise, such as the Hitchin--Thorpe inequality 
\[2\chi(M)\geq 3|\tau(M)|\]
in the oriented case. Our aim is to determine whether the same type of fundamental-group instability can persist in this setting, and in particular among Ricci-flat K\"ahler $4$-manifolds.

The common limit of two sequences exhibiting fundamental-group instability must necessarily be singular. Indeed, if $X$ were a smooth $n$-dimensional Riemannian manifold, then a theorem of Cheeger--Colding \cite[Theorem~A.1.12]{CheegerColding1997} (see also \cite{Colding1997Volume}) would imply that both $M_i$ and $N_i$ are diffeomorphic to $X$ for all sufficiently large $i$. On the other hand, the four-dimensional Einstein compactness theory \cite{Nakajima1988,BandoKasueNakajima1989,Anderson1989Einstein,Tian1990Calabi,CheegerNaber2015} implies that non-collapsing sequences of Ricci-flat $4$-manifolds with uniformly bounded diameter converge, up to a subsequence, to Ricci-flat Riemannian orbifolds with at most finitely many isolated singularities. The geometry developing near these singularities can be further described in terms of Ricci-flat ALE spaces, organized into a bubble tree \cite{Bando1990}. Thus, to exhibit fundamental-group instability in this setting, we are led to seek an orbifold that admits geometrically convergent desingularizations of different topological types.

The limiting geometry in Brena's construction cannot occur under uniform two-sided Ricci curvature bounds. The tangent cones of $\Sigma(cS^3/\mu_4)$ at its suspension points are given by $C(cS^3/\mu_4)$, which, for $c\neq1$, are not Ricci-flat, whereas tangent cones arising from noncollapsed limits with uniform two-sided Ricci curvature bounds are Ricci-flat on their regular part. There remains the exceptional value $c=1$, for which \(\Sigma(S^3/\mu_4)\) is a spherical Einstein orbifold. This limiting geometry is also ruled out in the Einstein setting. Huang--Ozuch \cite{HuangOzuch} prove that no spherical or hyperbolic Einstein $4$-orbifold with isolated singularities can be approximated in the Gromov--Hausdorff sense by smooth Einstein $4$-manifolds; see also \cite{OzuchII} and \cite{Ozuch2024Integrability} for earlier desingularization obstructions in related settings. Consequently, \(\Sigma(S^3/\mu_4)\) cannot arise as a nontrivial Gromov--Hausdorff limit of smooth Einstein $4$-manifolds.

The preceding discussion shows that an Einstein analogue of the instability construction obtained in \cite{Brena2026} would have to involve a genuinely different limiting geometry. Here, we show that such instability does occur in the Ricci-flat setting, even within the K\"ahler category. We construct two families of closed Ricci-flat K\"ahler $4$-manifolds with non-isomorphic fundamental groups converging in the Gromov--Hausdorff sense to the same compact flat orbifold.

\begin{theorem}\label{thm:main1}
There exist a compact flat four-dimensional Riemannian orbifold $X$, constants $\varepsilon_0,\mathrm v,C>0$, and families of closed $4$-manifolds
\[
(M,g_\varepsilon),
\qquad
(N,h_\varepsilon),
\qquad
0<\varepsilon<\varepsilon_0,
\]
such that
\[
\Ric_{g_\varepsilon}=\Ric_{h_\varepsilon}=0,
\qquad
\operatorname{vol}_{g_\varepsilon}(M),\operatorname{vol}_{h_\varepsilon}(N)\ge \mathrm v,
\qquad
\operatorname{diam}_{g_\varepsilon}(M),\operatorname{diam}_{h_\varepsilon}(N)\le C,
\]
and
\[
(M,g_\varepsilon)\xrightarrow{GH}X,
\qquad
(N,h_\varepsilon)\xrightarrow{GH}X
\qquad
\text{as }\varepsilon\to0.
\]
Moreover,
\[
\pi_1(M)\cong\mathbb Z_2,
\qquad
\pi_1(N)=1.
\]
$M$ is an Enriques manifold and $N$ is a $K3$ manifold. In particular, both families are K\"ahler and locally hyperk\"ahler.
\end{theorem}

Thus \Cref{thm:main1} realizes the same pair of fundamental groups as \cite{Brena2026}, namely
\[
1,\qquad \mathbb Z_2,
\]
but now within the compact Ricci-flat K\"ahler category. A classical theorem of Hitchin \cite{Hitchin1974} shows that, among compact nonflat locally hyperk\"ahler $4$-manifolds, the only possible fundamental groups are
\[
1,\qquad
\mathbb Z_2,\qquad
\mathbb Z_2\times\mathbb Z_2,
\]
corresponding respectively to a $K3$ surface, an Enriques surface, and a $\mathbb Z_2\times\mathbb Z_2$ quotient of $K3$. Using the same techniques, we also construct two locally hyperk\"ahler Ricci-flat $4$-manifolds with fundamental groups $\mathbb Z_2$ and $\mathbb Z_2\times\mathbb Z_2$ converging to the same compact flat orbifold. Taken together, these constructions exhibit fundamental-group instability involving all three types in Hitchin's classification.

\begin{theorem}\label{thm:main2}
There exist a compact flat four-dimensional Riemannian orbifold $X$, constants $\varepsilon_0,\mathrm v,C>0$, and families of closed Ricci-flat $4$-manifolds
\[
(P,g_\varepsilon),
\qquad
(M,h_\varepsilon),
\qquad
0<\varepsilon<\varepsilon_0,
\]
such that
\[
\operatorname{vol}_{g_\varepsilon}(P),
\operatorname{vol}_{h_\varepsilon}(M)\ge\mathrm v,
\qquad
\operatorname{diam}_{g_\varepsilon}(P),
\operatorname{diam}_{h_\varepsilon}(M)\le C,
\]
and
\[
(P,g_\varepsilon)\xrightarrow{GH}X,
\qquad
(M,h_\varepsilon)\xrightarrow{GH}X
\qquad
\text{as }\varepsilon\to0.
\]
Their fundamental groups are
\[
\begin{array}{lll}
&
\pi_1(P)\cong\mathbb Z_2\times\mathbb Z_2,
&
\pi_1(M)\cong\mathbb Z_2.
\end{array}
\]
The manifolds are locally hyperk\"ahler.
\end{theorem}

These examples also immediately produce higher-dimensional ones, since taking products with a fixed flat $m$-torus preserves Ricci-flatness and Gromov--Hausdorff convergence, and multiplies the fundamental group by $\mathbb Z^m$.

\begin{corollary}
\label{cor:higher-dimensional}
For every $n=4+m\ge4$, there exist two non-collapsing sequences of closed Ricci-flat $n$-manifolds with uniformly bounded diameter converging to the same $n$-dimensional Gromov--Hausdorff limit and having fundamental groups
\[
\mathbb Z^m \qquad\text{and}\qquad \mathbb Z_2\times\mathbb Z^m.
\]
There also exist such sequences with fundamental groups
\[
\mathbb Z_2\times\mathbb Z^m
\qquad\text{and}\qquad
\mathbb Z_2\times\mathbb Z_2\times\mathbb Z^m.
\]
\end{corollary}

The examples constructed above raise the question of whether non-collapsing sequences representing all three nonflat Ricci-flat types in Hitchin's classification can converge to the same orbifold. The following corollary answers this question in the negative.

\begin{corollary}
\label{cor:two-type-obstruction}
Non-collapsing sequences of closed, connected, nonflat, locally hyperk\"ahler $4$-manifolds representing all three fundamental-group types
\[
1,\qquad
\mathbb Z_2,\qquad
\mathbb Z_2\times\mathbb Z_2
\]
cannot converge to the same orbifold $X$. 
Moreover, whenever two distinct types converge to the same orbifold, they necessarily induce opposite orientations on $X$, their positive indices determine both $b_2^+(X)$ and $b_2^-(X)$, and $X$ must be flat.
\end{corollary}

The obstruction behind this phenomenon is a stability property of the positive index under Gromov--Hausdorff convergence. To put this property into context, recall that the results of Santos-Rodr\'iguez and Zamora \cite{SantosRodriguezZamora2023}, together with \cite{SormaniWei2001,PanWei2022}, imply that if a sequence $(M_i)$ satisfying the geometric bounds \eqref{eq:geometric-bounds} converges in the Gromov--Hausdorff sense to a compact metric space $X$, then
\[
b_1(M_i)=b_1(X)
\]
for all sufficiently large $i$.

Thus, despite the possible instability of the fundamental group, some topological invariants remain stable under Gromov--Hausdorff convergence. In the oriented setting (see \Cref{def:oriented-convergence} for the notion of induced orientation on a limit), we establish a corresponding stability property for the positive index.

\begin{theorem}[Positive-index stability]\label{thm:positive-index-stability}
Let $(Y_i,g_i,o_i)$ be a sequence of closed, connected, oriented locally Kähler $4$-manifolds satisfying 
\[
-\Lambda g_i\leq\operatorname{Ric}(g_i)\leq\Lambda g_i,
\qquad
\diam(Y_i,g_i)\leq D,
\qquad
\vol (Y_i,g_i)\geq \mathrm v>0, 
\] 
and converging in the Gromov--Hausdorff sense to the Riemannian orbifold $(X,g_X)$ with induced orientation $o_X$. Then
\[
b_2^+(Y_i,o_i)=b_2^+(X,o_X)
\]
for all sufficiently large $i$.
\end{theorem}

These stability phenomena are related to the behavior of harmonic forms and the spectrum of the Hodge Laplacian under Gromov--Hausdorff convergence; see, for instance, \cite{Honda2017Spectral,HondaMondino2025Poincare}. We return to this connection in \Cref{rmk:stability}.

\subsection{Structure of the Article}

The article is organized as follows. In \Cref{sec:method} we isolate the quotient--desingularization mechanism underlying the examples and describe the analytic methods used to produce the Ricci-flat metrics. The following section proves \Cref{thm:main1}, giving the pair $(1,\mathbb Z_2)$. \Cref{sec:thm2-first-construction} establishes \Cref{thm:main2}, giving the pair $(\mathbb Z_2,\mathbb Z_2\times \mathbb Z_2)$. Finally, \Cref{sec:positive-index-stability} establishes the stability of the positive index (\Cref{thm:positive-index-stability}), and derives its consequences for locally hyperk\"ahler 4-manifolds.

\section*{Acknowledgments}
The authors would like to thank Aaron Naber for helpful comments during the early stages of this project. AI tools were used to assist in polishing and editing the manuscript and in conducting bibliographic searches.

\bigskip

\section{The Instability Mechanism }\label{sec:method}

The purpose of this section is to describe the mechanism underlying the fundamental-group instability in our examples and to fix notation for the remainder of the article.

\subsection{The Kummer Orbifold and the Quotient Mechanism}

The basic mechanism starts from a Riemannian orbifold $X_0$ together with a finite group $G$ acting freely on its regular part, while some nontrivial elements of $G$ fix some singular points. The starting point in both of our constructions is the Kummer orbifold, which we now fix once and for all.

Let
\[
 \mathbb T^4:=\mathbb H/\Lambda,
\qquad
\Lambda=\mathbb Z\langle 1,i,j,k\rangle,
\]
and consider
\begin{equation}
X_0:=\mathbb T^4/\{\pm1\}
   =(\mathbb H/\Lambda)/\{\pm\mathrm{Id}\}.
\label{eq:KummerOrbifold}
\end{equation}
We equip $X_0$ with the flat orbifold metric $g_0$ induced by the Euclidean metric on $\mathbb H$. Its singular set, which we denote by $\mathcal S(X_0)$, consists of the sixteen $2$-torsion points of $\mathbb T^4$,
\[
\mathcal S(X_0)
=
\tfrac12\Lambda/\Lambda,
\]
each locally modeled on
\[
\mathbb H/\{\pm\mathrm{Id}\}
\cong
\mathbb R^4/\{\pm\mathrm{Id}\}.
\]

We then consider a finite group $G$ of isometries of $X_0$ which acts freely on the regular part of $X_0$, but may have nontrivial stabilizers on $\mathcal S(X_0)$, and form the quotient
\[
X:=X_0/G.
\]
We continue to denote by $g_0$ the induced flat orbifold metrics on the quotients of $X_0$ considered throughout the article.
The precise group action will vary between the examples and will be specified in the corresponding construction. The important point is that singular points of $X_0$ with nontrivial $G$-stabilizer acquire a larger local isotropy group after passing to $X$, whereas singular points lying in free $G$-orbits retain the original Kummer local model.

The existence of fixed points is the reason behind the main feature of the constructions: taking a finite quotient and resolving orbifold singularities need not commute, and therefore may provide two different manifolds.

\subsection{The First Sequence}

We first describe the construction in which one resolves $X_0$ before taking the quotient. To this end, we construct a covering space of the first manifold by desingularizing $X_0$ with respect to a suitable complex structure.

More precisely, let $J$ be a linear complex structure on $\mathbb H\simeq\mathbb R^4$ which is orthogonal with respect to the chosen flat metric.  The flat orbifold metric on $X_0$ is then K\"ahler with respect to $J$. Let $\omega_J$ be its standard K\"ahler form.

Fix $t>0$ sufficiently small so that the orbifold balls $B_t(p)$, $p\in\mathcal S(X_0)$, are pairwise disjoint. For each $p\in\mathcal S(X_0)$, choose $J$-holomorphic coordinates giving a biholomorphism
\[
\Psi_p:
B_t(p)
\longrightarrow
B_t(0)\subset\mathbb C^2/\{\pm\mathrm{Id}\}.
\]
We denote by $r_p: X \to \mathbb R$ the distance $r_p(x)=d_{g_0}(x,p)$. Let
\[
\varpi:T^*\mathbb CP^1
\longrightarrow
\mathbb C^2/\{\pm\mathrm{Id}\}
\]
denote the minimal resolution of the $A_1$ singularity. For each
$p\in\mathcal S(X_0)$, set
\[
U(p,t):=\varpi^{-1}(B_t(0)).
\] 
We denote by $E_p\subset U(p,t)$ the exceptional curve $\varpi^{-1}(0)\cong\mathbb{CP}^1$. When the radius $t$ is fixed, we write $U_p:=U(p,t)$. Composing with the local coordinates gives
\begin{equation}
\varpi_{p,J}:=\Psi_p^{-1}\circ\varpi:
U(p,t)\longrightarrow B_t(p).
\label{eq:localresol}
\end{equation}
The complex desingularization is obtained by gluing each $U(p,t)$ to
$X_0\setminus\mathcal S(X_0)$ via $\varpi_{p,J}$:
\begin{equation}
X_0^J
:=
\bigl(X_0\setminus\mathcal S(X_0)\bigr)
\mathbin{
  \mathop{\bigsqcup}\limits_{\varpi_{p,J}}
}
U(p,t),
\label{eq:resol}
\end{equation}
where $p$ ranges over $\mathcal S(X_0)$. Thus each
$y\in U(p,t)\setminus E_p$ is identified with
$\varpi_{p,J}(y)\in X_0\setminus\mathcal S(X_0)$. The complex structure $J$ on the regular part of $X_0$ and the standard complex structures on the
resolution pieces are compatible under these identifications and therefore define a global complex structure on $X_0^J$, which we continue to denote
by $J$. The resulting resolution map
\[
\pi_J:X_0^J\longrightarrow X_0
\]
is holomorphic and biholomorphic away from the exceptional divisor.

In the present setting, $X_0^J$ is the classical Kummer surface and hence a $K3$ surface. In each example, the group action and the complex structure $J$ will be chosen so that the $G$-action on $X_0$ lifts to an action
\[
G\curvearrowright X_0^J
\]
which is free.

The extension is easy to describe at the level of the resolution pieces. If an element of $G$ exchanges two singular points of $X_0$, its lift exchanges the corresponding resolution pieces. On the other hand, if a singular point is stabilized, the corresponding resolution piece is preserved and the induced action on its exceptional divisor must be examined. The complex structures in the examples below are chosen precisely so that the resulting action on $X_0^J$ is free. We then define
\[
M:=X_0^J/G.
\]
The metric is produced by the classical Kummer desingularization procedure. One interpolates the flat orbifold metric with suitably rescaled Eguchi--Hanson metrics \cite{EguchiHanson1978} near the exceptional curves. The resulting approximate metric is then perturbed to an exact Ricci-flat metric, and the estimates controlling the perturbation imply that the exact metric remains Gromov--Hausdorff close to $X_0$.

This construction goes back to Page \cite{Page1978K3} and Gibbons--Pope \cite{GibbonsPope1978}, and was rigorously realized in the 1990s by Kobayashi \cite{Kobayashi1990Moduli} and LeBrun--Singer \cite{LeBrunSinger}; see also the interpretation of Kobayashi--Todorov \cite{KobayashiTodorov1987} in terms of the period map. Further gluing constructions were given by Arezzo--Pacard \cite{ArezzoPacard2006}, Donaldson \cite{Donaldson2012}, and Jiang \cite{Jiang2025Kummer}.

For the purposes of the present argument, the gluing is carried out compatibly with the lifted $G$-action. Therefore it descends to a Ricci-flat metric on $M=X_0^J/G$. The resulting quotient
is Gromov--Hausdorff close to
\[
X=X_0/G.
\]

\subsection{The Second Sequence}

For the second family one reverses the order of the two operations. The central orbifold for both constructions,
\[
X=X_0/G,
\]
will be an appropriate quotient of $X_0$ endowed with its flat Riemannian metric.

The singular points of the quotient arise in two different ways. Singular points of $X_0$ lying in free $G$-orbits are identified along their orbits
and retain the local Kummer model
\[
\mathbb H/\{\pm\mathrm{Id}\}.
\]
On the other hand, if a singular point of $X_0$ has nontrivial stabilizer, then its image in the quotient has a larger local isotropy group.

In the examples considered here, the action on $X_0$ is induced by a finite group of isometries acting on the torus $\mathbb T^4=\mathbb H/\Lambda$.
If $\widetilde G$ denotes the corresponding group upstairs, then $\{\pm\mathrm{Id}\}\subset\widetilde G$ and there is an exact sequence
\[
1\longrightarrow\{\pm\mathrm{Id}\}
\longrightarrow\widetilde G
\longrightarrow G
\longrightarrow1.
\]
More generally, the local group at the image of a singular point is obtained from the subgroup of $\widetilde G$ stabilizing a lift of that point.

We then choose a second orthogonal linear complex structure $\widetilde J$ on $\mathbb H$ for which the local orbifold groups act through finite subgroups of $\mathrm{SU}(2,\widetilde J)$. We do so in such a way that $\widetilde J$ descends to the quotient orbifold and its singularities are
of ADE type.

Taking the corresponding minimal complex resolution at every singular point produces a smooth complex surface $X^{\widetilde J}$ together with a holomorphic map
\[
\pi_{\widetilde J}:X^{\widetilde J}\longrightarrow X,
\]
which is biholomorphic away from the exceptional divisor. Note that, unlike the previous case, there will be $A_3$ singularities. Locally the map $\pi_{\widetilde{J}}$ resolves these by adding a chain of three $(-2)$-curves. In the direct
version of the construction we set
\[
N:=X^{\widetilde J}.
\]

Following \cite{Hattori2026Harmonic}, we then equip $X^{\widetilde J}$ with a K\"ahler metric obtained by interpolating the flat orbifold metric with suitably rescaled ALE Ricci-flat K\"ahler metrics on the minimal resolutions of its ADE singularities. The existence and asymptotic properties of these ALE models follow from Kronheimer \cite{Kronheimer1989ALE}; see also  \cite[Theorem~8.2.3]{Joyce2000} and \cite[Theorem~3.3]{Joyce2001ALE}. Carrying out the gluing at sufficiently small scale yields a K\"ahler metric which is Gromov--Hausdorff close to the original flat orbifold. The final step is to perturb this approximate metric to an exact Ricci-flat metric.

In the proof of \Cref{thm:main2}, the second family is instead constructed through an intermediate orbifold rather than by resolving $X=X_0/G$ directly. The underlying mechanism is the same: quotienting first changes the local isotropy groups, and one then resolves the resulting ADE singularities before taking any remaining free quotient.

\subsection{The Final Diagram and the Noncommutative Behavior}

The estimates in the perturbation step of the two constructions above ensure that the resulting Ricci-flat manifolds are Gromov--Hausdorff close to the same quotient orbifold. In the direct construction used for \Cref{thm:main1}, this is summarized by the diagram in \Cref{fig:genscheme}. For \Cref{thm:main2}, the corresponding construction involves an intermediate orbifold before the resolution step. The corresponding diagram is given in \Cref{sec:thm2-first-construction}; see \Cref{fig:diagram2}.
\begin{figure}[htbp]
\centering
\begin{tikzpicture}

  \node (T4)  at (0,3.2) {$\mathbb H/\Lambda$};
  \node [anchor=east] at (-0.4,3.3) {$\widetilde G\curvearrowright$};

  \node (X0J) at (-2.8,1.5) {$X_0^J$};
  \node (X0)  at (0,1.5) {$X_0$};

  \node (M)   at (-2.8,-0.5) {$M$};
  \node (X)   at (0,-0.5) {$X$};
  \node (XtJ) at (3.0,-0.5) {$X^{\widetilde J}=N$};

  \draw[->]
    (T4) -- node[right] {$/\{\pm\mathrm{Id}\}$} (X0);

  \draw[->]
    (X0J) --
    node[above] {$\pi_J$}
    node[below, font=\scriptsize] {$G$-equivariant}
    (X0);

  \draw[->]
    (X0J) -- node[left] {$/G$ {\scriptsize (free)}} (M);
    
  \draw[->]
  (M) -- (X);
    
  \draw[->]
    (X0) -- node[right] {$/G$} (X);

  \draw[->]
    (XtJ) -- node[above] {$\pi_{\widetilde J}$} (X);

  \node at (-1.4,-0.9) {$\stackrel{\varepsilon\text{-}\mathrm{GH}}{\approx}$};
  \node at (1.5,-0.9) {$\stackrel{\varepsilon\text{-}\mathrm{GH}}{\approx}$};

\end{tikzpicture}
\caption {The two constructions leading to the same quotient orbifold $X$. On the left, one first resolves $X_0$ and then takes the free quotient by $G$; on the right, one first takes the quotient and then resolves. The $\varepsilon$--GH labels indicate the corresponding Gromov--Hausdorff approximations.}\label{fig:genscheme}
\end{figure}

In both constructions, the noncommutativity between taking the quotient and resolving the singularities comes from the presence of nontrivial stabilizers at singular points of $X_0$. Passing to the quotient changes the corresponding local isotropy groups before the resolution is performed. Consequently, the local resolutions appearing in the second sequence need not agree with the quotients of the $A_1$ resolutions appearing in the first sequence. Globally, this may lead to smooth manifolds with different fundamental groups even though both families converge to the same orbifold $X$. This is precisely what distinguishes the present constructions from more common settings, in which the relevant action is free and holomorphic. In that case, taking the quotient and passing to the minimal resolution commute.

We emphasize that the mechanism described above admits some flexibility. In more general constructions one may vary the lattice
$\Lambda$, the orbifold models, and the group action $G\curvearrowright X_0$. In the present work, however, our goal is to
produce Ricci-flat manifolds by perturbing approximate K\"ahler metrics. This imposes additional constraints on the construction. In particular,
the singularities must be compatible with the chosen complex structures and admit suitable complex resolutions, so that the desingularized spaces remain
within the K\"ahler framework required for the Calabi--Yau perturbation. If one were interested in examples satisfying only uniform two-sided Ricci curvature bounds, these complex-geometric restrictions could be relaxed considerably. For Einstein metrics outside the K\"ahler category, the complex-geometric restrictions disappear but are replaced by different desingularization constraints.

For the Calabi--Yau desingularization we follow the presentation in \cite{Hattori2026Harmonic}, which adapts the arguments in \cite{Jiang2025Kummer} to orbifold singularities modeled on $\mathbb C^2/\Gamma$, with $\Gamma<\mathrm{SU}(2)$ finite, thereby including all ADE singularities.


\section{Fundamental Group Instability for Ricci-Flat Kähler $4$-manifolds}\label{sec:provingthm1}

In this section we prove \Cref{thm:main1} following the strategy introduced in \Cref{sec:method}. We first specify the action on the Kummer orbifold and describe the resulting central quotient. In the following subsections we construct the two families of metrics. 

\subsection{The Action and the Central Orbifold}

We consider the action
\[
G=\mathbb Z_2\curvearrowright X_0
\]
induced by left quaternionic multiplication by $i$ on $\mathbb{H}$,
\[
L_i(q)=iq.
\]
Indeed,
\begin{equation}\label{eq:Z_2involution1}
L_i\circ(-\mathrm{Id})=(-\mathrm{Id})\circ L_i, \qquad L_i\Lambda=\Lambda, \qquad L_i^2=-\mathrm{Id},
\end{equation}
so $L_i$ descends to an involution of $X_0$ and hence defines a $\mathbb Z_2$-action.

The action is free on the regular part of $X_0$ and preserves $\mathcal S(X_0)$. More precisely, it fixes exactly four of the sixteen Kummer points and permutes the remaining twelve in pairs. The fixed points are
\begin{equation}
\operatorname{Fix}_{X_0}(G)
=
\left\{
[0],\,
\left[\frac{1+i}{2}\right],\,
\left[\frac{j+k}{2}\right],\,
\left[\frac{1+i+j+k}{2}\right]
\right\}.
\label{eq:fixpointsLi}
\end{equation}

\noindent We now introduce the central orbifold of the construction,
\[
X:=X_0/G.
\]
The four singular points of $X_0$ fixed by $G$ descend to four singular points of $X$ with isotropy group $\mathbb Z_4=\langle L_i\rangle$. The remaining twelve singular points form six $G$-orbits and give six singular points with local orbifold group \(\{\pm\mathrm{Id}\}\cong\mathbb Z_2\). Thus, $X$ has four singularities modeled on $\mathbb H/\langle L_i\rangle$ and six modeled on $\mathbb H/\{\pm\mathrm{Id}\}$. 

We equip $X$ with the flat orbifold metric induced by the Euclidean metric on $\mathbb H$.

\subsection{The Sequence $(M,g_\varepsilon)$}

Here we carry out the construction of the family of manifolds $(M,g_\varepsilon)$ appearing in \Cref{thm:main1}. Recall that we first construct the covering space, which will be $X_0^J$ for a suitable choice of complex structure $J$, and then obtain $M$ as its quotient; see \Cref{fig:genscheme}.

We say that $J$ is an \emph{orthogonal linear complex structure} on $X_0$ if it is induced by a linear map
\( J:\mathbb H\longrightarrow\mathbb H \) satisfying \( J^2=-\mathrm{Id}\) and orthogonal with respect to the Euclidean metric. The flat orbifold metric on $X_0$ is then K\"ahler with respect to $J$.

\begin{lemma}\label{lemma:example1_interm}
There exist an orthogonal linear complex structure $J$ on $X_0$, a complex resolution
\[
\pi_J:X_0^J\longrightarrow X_0,
\]
and constants $\varepsilon_0,C>0$ with the following properties.

\begin{enumerate}[topsep=0pt,label=\arabic*.]
    \item  $X_0^J$ is a Kummer surface and, in particular, a $K3$ surface. Hence $X_0^J$ is simply connected.

    \item The $\mathbb Z_2$-action on $X_0$ is antiholomorphic with respect to $J$ and extends to a free antiholomorphic action on $X_0^J$.

    \item For every $0<\varepsilon<\varepsilon_0$, $X_0^J$ admits a $\mathbb Z_2$-invariant Ricci-flat K\"ahler metric
    $g_\varepsilon$.

    \item  For every $0<\varepsilon<\varepsilon_0$, the resolution map
    \[
    \pi_J:(X_0^J,g_\varepsilon)\longrightarrow(X_0,g_0)
    \]
    is an equivariant $C\varepsilon^{1/2}$-Gromov--Hausdorff approximation.
\end{enumerate}
\end{lemma}
\begin{proof} \mbox{} \\
\noindent\textit{Complex structure and topology:} Recall that our action is given by left quaternionic multiplication by $i$, namely $L_i$. Choose $J := L_j$, the orthogonal linear complex structure on $\mathbb H$ given by quaternionic multiplication by $j$. Note that the complex structure descends to $X_0$. Moreover, $L_i L_j = -L_j L_i$ and therefore the $\mathbb{Z}_2$ action is antiholomorphic with respect to $J$. Let $X_0^J$ be the minimal complex resolution of $X_0$, as in \eqref{eq:resol}. The surface $X_0^J$ is the classical Kummer surface, hence a K3 surface and, in particular, is simply connected; see, for instance, \cite[Theorem~2.14]{Hattori2026Harmonic}.

\noindent\textit{Extending the action:}
For each $p\in\mathcal S(X_0)$, set $p'=L_i(p)$. Since $L_i$ is a $J$-antiholomorphic isometry, it restricts to an antiholomorphic isomorphism $L_i:B_t(p)\longrightarrow B_t(p')$. By uniqueness of the minimal resolution, applied to the conjugate complex structures, this map admits a unique antiholomorphic lift $\widetilde L_{i,p}:U_p\longrightarrow U_{p'}$, namely, it satisfies $\varpi_{p',J}\circ\widetilde L_{i,p} = L_i\circ\varpi_{p,J}$. In particular, $\widetilde L_{i,p}(E_p)=E_{p'}$, and on $U_p\setminus E_p$ the lift agrees with $\varpi_{p',J}^{-1}\circ L_i\circ\varpi_{p,J}$. These local lifts are therefore compatible with the gluing identifications in \eqref{eq:resol} and, together with the original action on $X_0\setminus\mathcal S(X_0)$, define a global antiholomorphic automorphism
\[
\widetilde L_i:X_0^J\longrightarrow X_0^J,
\qquad
\pi_J\circ\widetilde L_i=L_i\circ\pi_J.
\]
Since the induced map $L_i$ on $X_0$ is an involution, uniqueness of the
lifts implies that $\widetilde L_i^2=\mathrm{Id}$. Thus the original
$\mathbb Z_2$-action extends to $X_0^J$. 

\medskip
\noindent\textit{Freeness of the extended action:}
The extended action is free away from the exceptional divisor, since there it is conjugate to the original action on the regular part of $X_0$. If $L_i(p)\neq p$, then the two exceptional curves $E_p$ and $E_{L_i(p)}$ are exchanged, so neither contains a fixed point. It therefore remains to examine $E_p$ when $L_i(p)=p$.

Write \(\mathbb C_j=\mathbb R\oplus j\mathbb R\) and identify $(\mathbb H,J)$ with $\mathbb C_j^2$ by using $J-$holomorphic coordinates given by 
\(
q=z_1+z_2i,\) where \(z_1,z_2\in\mathbb C_j.
\)
Since
\(
iz=\overline z\,i\) for \(z\in\mathbb C_j,
\)
left multiplication by $i$, $L_i$, is represented in these coordinates by the antilinear map
\[
A(z_1,z_2)=(-\overline z_2,\overline z_1),
\qquad
A^2=-\mathrm{Id}.
\]
Translating these coordinates to $p$ gives a $J$-holomorphic orbifold chart $\Psi_p:B_r(p)\longrightarrow B_r(0)\subset \mathbb C_j^2/\{\pm \mathrm{Id}\}$, in which the local action of $L_i$ is induced by $A$.  Recall that, in the notation of \Cref{sec:method}, $\Psi_p$ denotes the local $J$-holomorphic orbifold chart used to define the resolution near $p$.

To compute the action of the lift on $E_p$, we use the following model for the minimal resolution of $\mathbb C^2/\{\pm\mathrm{Id}\}$:
\[
Y:=\operatorname{Bl}_0(\mathbb C^2)/\langle\tau\rangle,
\qquad
\tau(z,\ell)=(-z,\ell),
\]
where
\[
\operatorname{Bl}_0(\mathbb C^2)
=
\{(z,\ell)\in\mathbb C^2\times\mathbb{CP}^1:z\in\ell\}.
\]
The map
\[
(z,\ell)\longmapsto(Az,A\ell)
\]
is antiholomorphic, commutes with $\tau$, and extends the antiholomorphic automorphism of $\mathbb C^2/\{\pm\mathrm{Id}\}$ induced by $A$. It therefore descends to $Y$ and, after identifying the corresponding local resolution piece with $U_p$, agrees with the unique lift constructed above.

The exceptional curve consists of the classes of $(0,\ell)$, and $\tau$ acts trivially on these points. Hence the restriction of $\widetilde L_{i,p}$ to $E_p\simeq\mathbb{CP}^1$ is given by
\[
[z]\longmapsto[Az].
\]
If $[z]$ were fixed, then $Az=\lambda z$ for some
$\lambda\in\mathbb C^\times$. Since $A$ is antilinear,
\[
-z=A^2z=A(\lambda z)
=\overline\lambda Az
=|\lambda|^2z,
\]
which is impossible. Thus the extended action has no fixed points on $E_p$. Together with the preceding cases, this proves that $\widetilde L_i$ defines a free antiholomorphic $\mathbb Z_2$-action on $X_0^J$. 

\noindent \textit{Ricci-flat Kähler metrics:} We shall use the construction of the K\"ahler metrics given in \cite{Hattori2026Harmonic}.\footnote{Although we work with the complex structure $J=L_j$, this does not affect the construction given in \cite{Hattori2026Harmonic}. In fact, the flat complex orbifold $(X_0,L_j)$ is biholomorphically isometric to the standard Kummer orbifold, and hence we may apply \cite[Sections~2.2--2.4]{Hattori2026Harmonic} after this identification.}  Let $\eta$ be a fixed Eguchi--Hanson K\"ahler form on $T^*\mathbb{CP}^1$, and let $\varphi$ be its K\"ahler potential near infinity, as in \cite[Theorem~2.8]{Hattori2026Harmonic}. For each $p\in\mathcal S(X_0)$, denote by $\eta_p$ and $\varphi_p$ the corresponding copies of $\eta$ and $\varphi$, respectively. By \cite[Proposition~2.10 and Remark~2.12]{Hattori2026Harmonic}, $\eta_p$ belongs to a one-parameter family of K\"ahler forms, which we parametrize here by $\varepsilon>0$ and denote by $\eta_{p,\varepsilon^2}$. Moreover,
\begin{equation}
    \eta_{p,\varepsilon^2}
    =
    \varepsilon^2 H_{\varepsilon^{-1}}^*\eta_p, \label{eq:etap}
\end{equation}
where $H_\lambda$ is a holomorphic $\mathbb C^*$-action satisfying
$\varpi\circ H_\lambda=\lambda\varpi$. In particular, $\varepsilon$ is the corresponding linear scale of the Eguchi--Hanson geometry: if $E_p$ denotes the exceptional sphere, then $\operatorname{diam}_{\eta_{p,\varepsilon^2}}(E_p) = \varepsilon\,\operatorname{diam}_{\eta_p}(E_p)$.  

A corresponding family of approximate K\"ahler forms on $X_0^J$ is constructed by gluing the K\"ahler potentials of $\eta_{p,\varepsilon^2}$ to the flat potential. More precisely, on the ALE end, and in fact on $T^*\mathbb{CP}^1\setminus\mathbb{CP}^1$, one may take $\varphi_{p,\varepsilon^2}:=\varepsilon^2 H_{\varepsilon^{-1}}^*\varphi_p$ as a K\"ahler potential for $\eta_{p,\varepsilon^2}$. The approximate form $\widehat\omega_\varepsilon$ is then defined to agree with the corresponding Eguchi--Hanson form on $U(p,\varepsilon^{1/2})$ and with the flat form $\pi_J^*\omega_J$ outside $\bigcup_{p\in\mathcal S(X_0)}U(p,2\varepsilon^{1/2})$. Fix a smooth cutoff function $\chi:[0,\infty)\to[0,1]$ satisfying $\chi\equiv1$ on $[0,1]$ and $\chi\equiv0$ on $[2,\infty)$. On each resolution piece, set
\[
\chi_{\varepsilon^{1/2}}(x):=
\chi\left(\frac{r_p(\pi_J(x))}{\varepsilon^{1/2}}\right).
\]
On the transition annulus $A(p,\varepsilon^{1/2}):=U(p,2\varepsilon^{1/2})\setminus\overline{U(p,\varepsilon^{1/2})}$, the two are interpolated at the level of K\"ahler potentials, namely
\begin{equation}
\widehat\omega_\varepsilon
=
\sqrt{-1}\,\partial\bar\partial
\left(
\chi_{\varepsilon^{1/2}}\,
\varphi_{p,\varepsilon^2}
+
(1-\chi_{\varepsilon^{1/2}})
\frac{\pi_J^*r_p^2}{2}
\right).\label{eq:gluingdata}
\end{equation}Thus, for every sufficiently small $\varepsilon>0$, $\widehat\omega_\varepsilon$ defines a K\"ahler form on $X_0^J$ which agrees with the corresponding Eguchi--Hanson K\"ahler form on $U(p,\varepsilon^{1/2})$ and with the flat K\"ahler form outside $\bigcup_{p\in\mathcal S(X_0)}U(p,2\varepsilon^{1/2})$, with the two forms interpolated on the transition annulus. We denote by $\widehat g_\varepsilon$ its associated K\"ahler metric. Below, we also use the weighted H\"older norms of \cite[Subsection~3.1]{Hattori2026Harmonic}, with respect to $\widehat g_\varepsilon$ and the scale function
\[
\sigma_\varepsilon(x)
:=
\max\left\{
\varepsilon,
\min\left\{
1,
d_{g_0}\bigl(\pi_J(x),\mathcal S(X_0)\bigr)
\right\}
\right\}.
\]
Thus, set
\[
\sigma_\varepsilon(x,y)
:=
\min\{\sigma_\varepsilon(x),\sigma_\varepsilon(y)\}.
\]
There exists $R>0$, independent of $\varepsilon$, such that, for $x,y$ satisfying $d_{\widehat g_\varepsilon}(x,y)<R\,\sigma_\varepsilon(x,y)$, there is a unique minimizing $\widehat g_\varepsilon$-geodesic $\gamma_{xy}$ joining $x$ to $y$. Let $P_{y\to x}:T_yX_0^J\longrightarrow T_xX_0^J$ denote parallel transport along $\gamma_{xy}$ from $y$ to $x$. With this convention, for a function $u$, we define the weighted H\"older norm by
\[
\begin{aligned}
\|u\|_{C^{k,\alpha}_{\delta,\widehat{g_\varepsilon}}}
:={}&
\sum_{j=0}^k
\sup_{X_0^J}
\sigma_\varepsilon^{-\delta+j}
\left|
\nabla_{\widehat g_\varepsilon}^j u
\right|_{\widehat g_\varepsilon}
\\
&\quad+
\sup_{\substack{x\neq y\\
d_{\widehat{g_\varepsilon}}(x,y)
<
R\,\sigma_\varepsilon(x,y)}}
\sigma_\varepsilon(x,y)^{-\delta+k+\alpha}
\frac{
\left|
\nabla_{\widehat g_\varepsilon}^k u(x)
-
P_{y\to x}
\bigl(\nabla_{\widehat g_\varepsilon}^k u(y)\bigr)
\right|_{\widehat g_\varepsilon(x)}
}{
d_{\widehat g_\varepsilon}(x,y)^\alpha
}.
\end{aligned}
\]
The weighted H\"older norms of tensor fields are defined analogously, using the corresponding covariant derivatives and induced parallel transport.

By \cite[Theorem~3.4]{Hattori2026Harmonic}, one can perturb $\widehat\omega_\varepsilon$ to a Ricci-flat K\"ahler form $\omega_\varepsilon$ by solving the complex Monge--Amp\`ere equation. More precisely, there exists $\varphi_\varepsilon\in C^\infty(X_0^J)$ such that
\begin{equation}
\omega_\varepsilon
=
\widehat\omega_\varepsilon
+
\sqrt{-1}\,\partial\bar\partial\varphi_\varepsilon,
\qquad
\omega_\varepsilon^2
=
\frac{W_\varepsilon}{2}\,
\Omega\wedge\overline{\Omega}.\label{eq:calabiyaueq}
\end{equation}
Here $\Omega$ denotes the fixed nowhere-vanishing holomorphic $2$-form on $X_0^J$ extending $\pi_J^*\Omega_0$, where, in the $J$-holomorphic coordinates $q=z_1+z_2i$ on $\mathbb H$ introduced above,
\(
\Omega_0 := dz_1 \wedge dz_2.
\)
This choice descends to the regular part of $X_0$ and is normalized so that
\(
\omega_0^2 = \frac 12 \Omega_0 \wedge \overline\Omega_0.
\)
The positive constant $W_\varepsilon$ is determined by the compatibility condition for \eqref{eq:calabiyaueq}, namely
\begin{equation}\label{eq:normalization_omega}
W_\varepsilon
:=
\frac{\displaystyle\int_{X_0^J}\widehat\omega_\varepsilon^2}
{\displaystyle\frac12\int_{X_0^J}\Omega\wedge\overline{\Omega}}.
\end{equation}
 Moreover, for every fixed $k\geq0$, $\alpha\in(0,1)$, and $\delta\in(-2,0)$, we have
\[
\|\varphi_\varepsilon\|_
{C^{k+2,\alpha}_{\delta,\widehat g_\varepsilon}}
\leq
C_{k,\alpha,\delta}\,\varepsilon^{3-\delta/2},
\]
and
\begin{equation}\label{eq:W}
|W_\varepsilon-1|
\leq C\varepsilon^4,
\end{equation}
where $C$ and $C_{k,\alpha,\delta}$ are independent of $\varepsilon$. Consequently,
\begin{equation}
\|\omega_\varepsilon-\widehat\omega_\varepsilon\|_
{C^{k,\alpha}_{\delta-2,\widehat g_\varepsilon}}
\leq
C'_{k,\alpha,\delta}\,
\varepsilon^{3-\delta/2},
\qquad -2<\delta<0,
\label{eq:estimatepert}
\end{equation}
where again the constant is independent of $\varepsilon$, but may depend on $k$, $\alpha$, and $\delta$.

On the transition annulus, \cite[Eq.~(4)]{Hattori2026Harmonic} gives $\left|\widehat g_\varepsilon-\pi_J^*g_0\right|_{\pi^*_Jg_0} \leq C\varepsilon^2$. Together with \eqref{eq:etap}, this gives $\operatorname{diam}_{\widehat g_\varepsilon} U(p,2\varepsilon^{1/2})\leq C\varepsilon^{1/2}$. Since $\widehat g_\varepsilon=\pi_J^*g_0$ outside these regions, it follows that the resolution map 
\[\pi_J:(X_0^J,\widehat g_\varepsilon)\to(X_0,g_0)\] is a $C\varepsilon^{1/2}$-Gromov--Hausdorff approximation. On the other hand, \eqref{eq:estimatepert} gives, for $-2<\delta<0$, \[\left|g_\varepsilon-\widehat g_\varepsilon \right|_{\widehat g_\varepsilon} \leq C_\delta\varepsilon^{1+\delta/2}.\] Choosing $\delta\in(-1,0)$, the latter is $o(\varepsilon^{1/2})$. Hence \[\pi_J:(X_0^J,g_\varepsilon)\to(X_0,g_0)\] is still a $C\varepsilon^{1/2}$-Gromov--Hausdorff approximation. 

\noindent \textit{The extended $\mathbb Z_2$-action is isometric:} 
Since $\widetilde L_i$ is antiholomorphic, $\widetilde L_i^* g_\varepsilon = g_\varepsilon$ is equivalent to $-\widetilde L_i^*\omega_\varepsilon = \omega_\varepsilon$. Moreover, the form $-\widetilde L_i^*\omega_\varepsilon$ is again a Ricci-flat K\"ahler form on $(X_0^J,J)$. Since
$$[\omega_\varepsilon] = [\widehat\omega_\varepsilon]$$ in $H^2(X_0^J;\mathbb R)$, we have $$[-\widetilde L_i^*\omega_\varepsilon] = [-\widetilde L_i^*\widehat\omega_\varepsilon].$$

Thus, by uniqueness of the Ricci-flat K\"ahler form in a fixed K\"ahler class, it suffices to prove that
\[
[-\widetilde L_i^*\widehat\omega_\varepsilon]
=
[\widehat\omega_\varepsilon].
\]

Indeed, on
\[
X_0^J\setminus
\bigcup_{p\in\mathcal S(X_0)}U(p,2\varepsilon^{1/2})
\]
we have
\(
\widehat\omega_\varepsilon=\pi_J^*\omega_0. \) Since $L_i$ is an antiholomorphic isometry of $(X_0,J,g_0)$,
\(
L_i^*\omega_0=-\omega_0.
\)
Using
\[
\pi_J\circ\widetilde L_i=L_i\circ\pi_J,
\]
it follows that
\[
-\widetilde L_i^*\widehat\omega_\varepsilon
=
\widehat\omega_\varepsilon
\]
on this region. Recall that $H_2(X_0^J,\mathbb{R})$ is generated by the six torus classes coming from the orbifold part together with the sixteen exceptional classes. The torus representatives may be chosen inside the region where
\[
-\widetilde L_i^*\widehat\omega_\varepsilon
=
\widehat\omega_\varepsilon,
\]
so the two forms have the same pairing with these six classes. It remains to consider the exceptional curves. Since $\widetilde L_i$ is antiholomorphic, its restriction to an exceptional curve reverses its complex orientation. Thus, if
\[
\widetilde L_i(E_p)=E_q,
\]
then
\[
\int_{E_p}-\widetilde L_i^*\widehat\omega_\varepsilon
=
\int_{E_q}\widehat\omega_\varepsilon.
\]
All exceptional curves have the same area for $\widehat\omega_\varepsilon$, namely
\[
\varepsilon^2c, \qquad c:=\int_{\mathbb{CP}^1}\eta.
\]
Hence the two forms have the same pairing with a generating set of $H_2(X_0^J,\mathbb{R})$, and consequently
\[
[-\widetilde L_i^*\widehat\omega_\varepsilon]
=
[\widehat\omega_\varepsilon].
\]
By the uniqueness argument above,
\[
-\widetilde L_i^*\omega_\varepsilon=\omega_\varepsilon,
\]
and therefore
\[
\widetilde L_i^*g_\varepsilon=g_\varepsilon.
\]
\end{proof}

\begin{remark}
In the previous proof, note the importance of the action being antiholomorphic. If it were holomorphic, at one of the fixed singular points the induced action on the corresponding exceptional divisor would have fixed points. If $L_i$ were neither holomorphic nor antiholomorphic, it would not lift to the chosen complex resolution.
\end{remark}

\begin{theorem}\label{thm:firstsequencethm1}
For every sufficiently small $\varepsilon>0$, there exists a Ricci-flat metric $g_\varepsilon$ on
\[
M:=X_0^J/\mathbb Z_2
\]
such that
\[
d_{GH}\bigl((M,g_\varepsilon),X\bigr)<C\varepsilon^{1/2},
\]
where $C>0$ is independent of $\varepsilon$. Moreover, there exists a complex structure $I_\varepsilon$ on $M$ for which $(M,I_\varepsilon)$ is an Enriques surface and $g_\varepsilon$ is K\"ahler. In particular, $(M,g_\varepsilon,I_\varepsilon)$ is a Ricci-flat K\"ahler surface.
\end{theorem}
\begin{proof}
By \Cref{lemma:example1_interm}, the free $\mathbb Z_2$-action on
$X_0^J$ acts by isometries of $g_\varepsilon$. Hence $g_\varepsilon$
descends to a Ricci-flat metric on
\[
M=X_0^J/\mathbb Z_2,
\]
which we continue to denote by $g_\varepsilon$. The equivariant
$C\varepsilon^{1/2}$-Gromov--Hausdorff approximation
\[
\pi_J:(X_0^J,g_\varepsilon)\longrightarrow(X_0,g_0)
\]
therefore descends to a $C\varepsilon^{1/2}$-Gromov--Hausdorff
approximation
\[
\pi:(M,g_\varepsilon)\longrightarrow X.
\]
This proves the Gromov--Hausdorff estimate.

It remains to endow $M$ with a complex structure for which $g_\varepsilon$ is K\"ahler. Note that the complex structure $J$ itself does not descend, since the $\mathbb Z_2$-action induced by $L_i$ is $J$-antiholomorphic. Set $\Omega_\varepsilon:=\sqrt{W_\varepsilon}\,\Omega$. By \eqref{eq:calabiyaueq}, we have
\[
\omega_\varepsilon^2
=
\frac12\,\Omega_\varepsilon
\wedge\overline{\Omega_\varepsilon}.
\]
Since $g_\varepsilon$ is Ricci-flat K\"ahler on the compact surface $X_0^J$, the holomorphic $2$-form $\Omega_\varepsilon$ is parallel. Together with the normalization above, this implies that the three parallel $2$-forms
\[
\omega_\varepsilon,\qquad
\operatorname{Re}\Omega_\varepsilon,\qquad
\operatorname{Im}\Omega_\varepsilon
\]
form a hyperk\"ahler triple compatible with $g_\varepsilon$. Let $I_\varepsilon$ denote the complex structure characterized by
\[
g_\varepsilon(I_\varepsilon\cdot,\cdot)
=
\operatorname{Re}\Omega_\varepsilon
=
\sqrt{W_\varepsilon}\,\operatorname{Re}\Omega.
\]
In the $J$-holomorphic coordinates used above, the lift $\widetilde L_i$ is induced by
\[
A(z_1,z_2)=(-\overline z_2,\overline z_1),
\]
and therefore it is easy to check that
\(
\widetilde L_i^*\Omega=\overline{\Omega}.
\) In particular,
\[
\widetilde L_i^*\operatorname{Re}\Omega
=
\operatorname{Re}\Omega.
\]
Since
\(
\widetilde L_i^*g_\varepsilon=g_\varepsilon,
\)
the defining relation for $I_\varepsilon$ gives
\[
\widetilde L_i^*I_\varepsilon=I_\varepsilon.
\]
Thus $I_\varepsilon$ descends to the quotient $M$, and the descended metric $g_\varepsilon$ is K\"ahler with respect to it. Finally, $X_0^J$ is a K3 surface and the $\mathbb Z_2$-action is free and holomorphic with respect to $I_\varepsilon$. Therefore, the quotient $(M,I_\varepsilon)$ is an Enriques surface.
\end{proof}

\subsection{The Sequence $(N,h_\varepsilon)$}

Here we carry out the construction of the family $(N,h_\varepsilon)$ appearing in \Cref{thm:main1} and   complete the proof of \Cref{thm:main1}. In the notation of \Cref{fig:genscheme}, this manifold plays the role of $X^{\widetilde J}$, for a suitable orthogonal linear complex structure $\widetilde J$ on the central orbifold $X$.

\begin{lemma}\label{thm:1-2}
There exist an orthogonal linear complex structure $\widetilde J$ on $X$, a complex resolution
\[
\pi_{\widetilde J}:N=X^{\widetilde J}\longrightarrow X,
\]
and constants $\varepsilon_0,C>0$ such that $N$ is a K3 surface. Moreover, for every $0<\varepsilon<\varepsilon_0$, $N$ admits a Ricci-flat K\"ahler metric $h_\varepsilon$ with respect to $\widetilde J$, satisfying
\[
d_{GH}\bigl((N,h_\varepsilon),X\bigr)<C\varepsilon^{1/2}.
\]
In particular, $N$ is simply connected.
\end{lemma}
\begin{proof} 
Recall that the action $L_i$ is induced by left quaternionic multiplication by $i$. Unlike in the previous construction, we now take the quotient before resolving the singularities, and therefore need to choose an orthogonal linear complex structure $\widetilde J$ which descends to the quotient $X=X_0/G$. This requires the action to be holomorphic with respect to $\widetilde J$, that is, $L_i\widetilde J=\widetilde J L_i$. Moreover, in order to apply the Calabi--Yau desingularization, we require the local isotropy groups to be contained in $\mathrm{SU}(2,\widetilde J)$. Although the choice $\widetilde J=L_i$ satisfies the first condition, the complex determinant of $L_i$ with respect to this complex structure is $-1$. Instead, we choose $\widetilde J=R_i$, given by right quaternionic multiplication by $i$. Since left and right quaternionic multiplications commute, the action is holomorphic and $\widetilde J$ descends to a complex structure on $X$.

Writing $q=z_1+jz_2$, with $z_1,z_2\in\mathbb C_i$, gives $\widetilde J$-holomorphic coordinates identifying $(\mathbb H,R_i)$ biholomorphically and isometrically with $\mathbb C^2$. In these coordinates, the action of $L_i$ is given by
\[
    (z_1,z_2)\longmapsto(iz_1,-iz_2).
\]
Thus $L_i\in SU(2,\widetilde J)$, and consequently the singularities of $X$ are of ADE type. More precisely, $X$ has four singularities of type $A_3$ and six singularities of type $A_1$.

Let $N=X^{\widetilde J}$ be the minimal complex resolution of $X$, obtained as in \eqref{eq:resol}, replacing neighborhoods of its $A_3$ and $A_1$ singularities by their corresponding minimal resolutions. Since $L_i \Lambda = \Lambda$, we may apply \cite[Theorem~2.14]{Hattori2026Harmonic}. Therefore $N$ is a $K3$ surface and, in particular, is simply connected.

Finally, applying the construction of \cite[Sections~2.2--2.4 and Theorem~3.4]{Hattori2026Harmonic} and arguing as in the proof of \Cref{lemma:example1_interm}, we obtain a Ricci-flat K\"ahler metric $h_\varepsilon$ on $N$. Denote its K\"ahler form by $\widetilde\omega_\varepsilon$. Let $\widetilde\omega_0$ be the flat K\"ahler form on the regular part of $X$ with respect to $\widetilde J$, and let $\widetilde\Omega_0$ be the flat holomorphic $2$-form defined in the $\widetilde{J}$-holomorphic coordinates  $q=z_1+jz_2$ on $\mathbb H$ by $\widetilde{\Omega}_0=dz_1\wedge dz_2$. Since $L_i^* \widetilde{\Omega}_0 = \widetilde{\Omega}_0$, it descends to the regular part of $X$, and
\[
\widetilde\omega_0^2 = \frac12\widetilde\Omega_0 \wedge\overline{\widetilde\Omega_0}.
\]
Let $\widetilde\Omega$ denote the nowhere-vanishing holomorphic $2$-form on $N$ extending $\pi_{\widetilde J}^*\widetilde\Omega_0$. Then
\begin{equation}\label{eq:W_tilde}
\widetilde\omega_\varepsilon^2 = \frac{\widetilde W_\varepsilon}{2}\,
\widetilde\Omega\wedge\overline{\widetilde\Omega}, \qquad |\widetilde W_\varepsilon-1| \leq C\varepsilon^4.
\end{equation}
Here $\widetilde W_\varepsilon$ is the normalization constant defined analogously to $W_\varepsilon$ in \eqref{eq:normalization_omega}. Moreover,
\[
d_{GH}\bigl((N,h_\varepsilon),X\bigr)< C\varepsilon^{1/2}.
\]
\end{proof}

\begin{remark}
The preceding coordinate description shows directly that the present construction is exactly the one considered in \cite{Hattori2026Harmonic}. Writing
$q=z_1+jz_2$ identifies $(\mathbb H,R_i)$ biholomorphically and isometrically with $\mathbb C^2$, identifies the lattice
$\Lambda=\mathbb Z\langle 1,i,j,k\rangle$ with the lattice used there, and represents $L_i$ by
\[
(z_1,z_2)\longmapsto(iz_1,-iz_2).
\]
Since
\[
X=X_0/G=\mathbb T^4/\langle L_i\rangle,
\]
the central orbifold $X$ is precisely the orbifold of
\cite[Example~8(2)]{Hattori2026Harmonic}.
\end{remark}

We now assemble the two constructions to prove \Cref{thm:main1}.

\begin{proof}[Proof of \Cref{thm:main1}]
By \Cref{thm:firstsequencethm1}, for every sufficiently small $\varepsilon>0$ there is a Ricci-flat K\"ahler metric $g_\varepsilon$ on the Enriques surface $(M,I_\varepsilon)$ such that
\[
d_{GH}\bigl((M,g_\varepsilon),X\bigr) <C\varepsilon^{1/2}.
\]
By \Cref{thm:1-2}, there is a Ricci-flat K\"ahler metric $h_\varepsilon$ on the $K3$ surface 
\[
N = X^{\widetilde J}
\]
such that
\[
d_{GH}\bigl((N,h_\varepsilon),X\bigr) <C\varepsilon^{1/2}.
\]
After decreasing $\varepsilon_0$ and enlarging $C$ if necessary, the same constants may be used for both families. Hence,
\[
(M,g_\varepsilon)\xrightarrow{GH}X, \qquad (N,h_\varepsilon)\xrightarrow{GH}X \qquad \text{as }\varepsilon\to0.
\]
The fundamental groups claimed follow directly from the two constructions: $N$ is a simply connected $K3$ surface, and the $\mathbb Z_2$-action defining $M$ is free. Thus, $\pi_1(M) \cong \mathbb Z_2.$

It remains only to verify the uniform geometric bounds appearing in the statement. Since both families converge to the compact space $X$, their diameters are uniformly bounded. The volume bounds follow, for instance, from the normalization constants in the Calabi--Yau equations. Indeed, \eqref{eq:calabiyaueq} and the normalization of $\Omega_0$ give
\[
\operatorname{vol}_{g_\varepsilon}(M) = W_\varepsilon\,\operatorname{vol}_{g_0}(X).
\]
Similarly,
\[
\operatorname{vol}_{h_\varepsilon}(N) = \widetilde W_\varepsilon\,\operatorname{vol}_{g_0}(X).
\]
Since, by  \eqref{eq:W} and \eqref{eq:W_tilde},
\[
|W_\varepsilon-1| + |\widetilde W_\varepsilon-1| \leq C\varepsilon^4,
\]
after decreasing $\varepsilon_0$ if necessary, both volumes are bounded
below by
\[
\frac12\,\operatorname{vol}_{g_0}(X)>0.
\]

Finally, $N$ is a Ricci-flat K\"ahler $K3$ surface and hence is hyperk\"ahler. The universal cover $X_0^J$ of $M$, equipped with
$g_\varepsilon$, is likewise a Ricci-flat K\"ahler $K3$ surface and hence hyperk\"ahler. 
\end{proof}

\section{Fundamental Group Instability: Enriques and Hitchin Manifolds}
\label{sec:thm2-first-construction}

In this section, we prove \Cref{thm:main2} following the scheme introduced in \Cref{sec:method}. We first describe the action of $\mathbb Z_2\times\mathbb Z_2$ on the Kummer orbifold $X_0$ and define the central orbifold $X$. We then give separate constructions of the two families $(P,g_\varepsilon)$ and $(M,h_\varepsilon)$.

\subsection{The Action and the Central Orbifold}
We consider an action
\[
G\cong\mathbb Z_2\times\mathbb Z_2\curvearrowright X_0
\]
generated by the maps induced by the affine isometries of $\mathbb H$
\begin{equation}
L_i(q)=iq
\end{equation}
and
\begin{equation}
D(q)=L_uR_{u^{-1}}(q)+t,
\qquad
u:=\frac{j+k}{\sqrt2},
\qquad
t:=\frac12(1+i+j+k).
\label{eq:thm2-LiD}
\end{equation}
The linear parts of $L_i$ and $D$ preserve the lattice $\Lambda$, so both maps descend to isometries of $\mathbb T^4$. Moreover, they normalize the central involution $\{\pm\mathrm{Id}\}$, and hence descend further to isometries of $X_0=\mathbb T^4/\{\pm\mathrm{Id}\}$. We continue to denote the induced maps by $L_i$ and $D$. On $X_0$ they satisfy
\[
L_i^2=D^2=\mathrm{Id},
\qquad
L_iD=DL_i.
\]
Thus they generate the action
\[
G:=\langle L_i,D\rangle
\cong \mathbb Z_2\times\mathbb Z_2
\curvearrowright X_0.
\]
The involution induced by $L_i$ is the one considered in \Cref{sec:provingthm1}: it acts freely on the regular part of $X_0$ and fixes precisely the four singular points listed in \eqref{eq:fixpointsLi}.

On the other hand, we have
\begin{equation}
\Fix(D)=\Fix(L_iD)=\varnothing.\label{eq:thm2-fixedsets}
\end{equation}
To see this, write $q=z_1+z_2j$, with $z_1,z_2\in\mathbb C_i$, and set $\xi=(1+i)/2$. In these coordinates,
\[
D(z_1,z_2)=(\overline z_1+\xi,i\overline z_2+\xi),
\qquad
L_iD(z_1,z_2)=(i\overline z_1+i\xi,-\overline z_2+i\xi).
\]
A fixed point on $X_0$ would require $D(q)\equiv\pm q$ or $L_iD(q)\equiv\pm q\pmod{\Lambda}$. The corresponding coordinate equations have no solutions. Hence
\[
\Fix(D)=\Fix(L_iD)=\varnothing.
\]
The sixteen singular points therefore decompose into two $G$-orbits of cardinality two, with stabilizer generated by the involution induced by $L_i$, and three free $G$-orbits of cardinality four. We define the central orbifold for this example by
\begin{equation}
X:=X_0/G.
\label{eq:X1-def}
\end{equation}
Accordingly, the two stabilized orbits give rise to two singular points locally
modeled on $\mathbb H/\langle L_i\rangle$, where $\langle L_i\rangle\cong\mathbb Z_4$ since $L_i^2=-\mathrm{Id}$ on $\mathbb{H}$. The three free orbits retain the original Kummer local model $\mathbb H/\{\pm\mathrm{Id}\}$. Thus $X$ has two singularities of the former type and three of the latter.

\subsection{The Sequence $(P,g_\varepsilon)$}
For the first sequence, we follow the scheme of \Cref{sec:method}: we resolve $X_0$ with respect to a suitable complex structure, extend the $G$-action freely to the resolution, construct $G$-invariant Ricci-flat K\"ahler metrics, and then take the quotient by $G$.

\begin{lemma}
\label{thm:thm2-first-P}
There exist an orthogonal linear complex structure $J$ on $X_0$, a complex resolution
\[
\pi_J:X_0^J\longrightarrow X_0,
\]
and constants $\varepsilon_0,C>0$ with the following properties.

\begin{enumerate}[topsep=0pt,label=\arabic*.]
\item The surface $X_0^J$ is a Kummer surface and, in particular, a $K3$ surface. Hence, $X_0^J$ is simply
connected.

\item The action of $G$ extends to a free action on $X_0^J$. With respect to $J$, $L_i$ acts antiholomorphically, whereas $D$ acts holomorphically.

\item For every $0<\varepsilon<\varepsilon_0$, $X_0^J$ admits a $G$-invariant Ricci-flat K\"ahler metric $g_\varepsilon$ such that
\[
\pi_J:(X_0^J,g_\varepsilon)
\longrightarrow (X_0,g_0)
\]
is a $G$-equivariant
$C\varepsilon^{1/2}$-Gromov--Hausdorff approximation.
\end{enumerate}
Consequently,
\begin{equation}
P:=X_0^J/G
\label{eq:P1-def}
\end{equation}
is a smooth manifold satisfying
\[
\pi_1(P)\cong\mathbb Z_2\times\mathbb Z_2,
\]
and the descended Ricci-flat metrics, which we also denote by $g_\varepsilon$, satisfy
\[
d_{GH}\bigl((P,g_\varepsilon),X\bigr)
<C\varepsilon^{1/2}.
\]
\end{lemma}
The proof below parallels those of \Cref{lemma:example1_interm} and \Cref{thm:firstsequencethm1}. We therefore present the argument more briefly, referring to those proofs for the corresponding details.
\begin{proof} \mbox{} \\
\noindent\textit{Complex structure and topology:} Pick the orthogonal linear complex structure given by $J = L_{u}$ on $\mathbb H$, where $u = \tfrac{j+k}{\sqrt 2}$, as in \eqref{eq:thm2-LiD}. Note that the complex structure descends to $X_0$. Moreover, $$L_i L_{\tfrac{j+k}{\sqrt 2}} = -L_{\tfrac{j+k}{\sqrt 2}} L_i,$$ and therefore the $L_i$ action is antiholomorphic with respect $J$. For $D$, only the linear part is relevant (see \eqref{eq:thm2-LiD}). Since $(L_uR_{u^{-1}})\circ L_u=L_u\circ(L_uR_{u^{-1}})$, the map $D$ is $J$-holomorphic.

As before, let $X_0^J$ be the minimal complex resolution of $X_0$, as in \eqref{eq:resol}. Recall that $X_0^J$ is the classical Kummer surface \cite[Theorem~2.14]{Hattori2026Harmonic}.

\noindent\textit{Extension and freeness of the extended action:} The holomorphic involution $D$ and the antiholomorphic involution $L_i$ lift uniquely through the minimal resolution. Indeed, away from the exceptional divisor the lifts are necessarily given by
\[
\pi_J^{-1}\circ D\circ\pi_J,
\qquad
\pi_J^{-1}\circ L_i\circ\pi_J,
\]
and the uniqueness of the minimal resolution, applied also to the conjugate complex structure in the antiholomorphic case, extends these maps uniquely across the exceptional curves; compare the proof of
\Cref{lemma:example1_interm}. The relations between the generators are preserved, so they define an action $$G \curvearrowright X_0^J.$$ 
We continue to use $D$ and $L_i$ to refer to the corresponding extended elements. Since $D$ and $DL_i$ have no fixed points on $X_0$, they do not stabilize any singular point and hence exchange the corresponding resolution pieces. Their lifts are therefore fixed-point free. For $L_i$, the same local argument as in the proof of \Cref{lemma:example1_interm} applies at each of the four fixed Kummer points: on the corresponding exceptional curve, the induced antiholomorphic map has square $-\mathrm{Id}$ before projectivization and hence has no fixed points. Thus the $G$-action on $X_0^J$ is free.

\medskip
\noindent\textit{Approximate and Ricci-flat K\"ahler metrics:}
We repeat verbatim the construction in \Cref{lemma:example1_interm}. We denote by $\widehat\omega_\varepsilon$ the approximate K\"ahler form obtained by the gluing construction, and by $\omega_\varepsilon$ the unique Ricci-flat K\"ahler form in the K\"ahler class $[\widehat\omega_\varepsilon] \in H^2(X_0^J,\mathbb{R})$. Thus, for $\varepsilon>0$ sufficiently small, we obtain a Ricci-flat K\"ahler metric $g_\varepsilon$ on $X_0^J$, with K\"ahler form $\omega_\varepsilon$, such that the resolution map
\[
\pi_J:(X_0^J,g_\varepsilon)\longrightarrow(X_0,g_0)
\]
is a $C\varepsilon^{1/2}$-GH approximation. Moreover, $\pi_J$ is
$G$-equivariant.

\noindent\textit{The extended $G$-action acts by isometries:} As in \Cref{lemma:example1_interm}, the action of $L_i$ is by isometries. For $D$, since it is holomorphic, $D^*g_\varepsilon = g_\varepsilon$ is equivalent to $D^*\omega_\varepsilon = \omega_\varepsilon$. By uniqueness of the Ricci-flat K\"ahler representative in a fixed K\"ahler class, it suffices to show that $$[D^*\hat\omega_\varepsilon] = [\hat\omega_\varepsilon]$$ in $H^2(X_0^J,\mathbb{R})$. 

By construction, on $X_0^J \backslash \bigcup_{p\in\mathcal S(X_0)}U(p,2\varepsilon^{1/2})$ we have $\widehat\omega_\varepsilon=\pi_J^*\omega_0$.
Since $D$ is a holomorphic isometry of $(X_0,J,g_0)$, we have
$D^*\omega_0=\omega_0$, and hence $D^*\widehat\omega_\varepsilon=\widehat\omega_\varepsilon$ on this region. Recall that $H_2(X_0^J,\mathbb R)$ is generated by the six torus
classes coming from the orbifold part together with the sixteen
exceptional classes. The torus representatives may be chosen inside
the region where $D^*\widehat\omega_\varepsilon=\widehat\omega_\varepsilon$, so the two forms have the same pairing with these six classes. 

It remains to consider the sixteen exceptional curves. Since $D$ is holomorphic, its restriction to an exceptional curve preserves its complex
orientation. Thus, if
\[
D(E_p)=E_q,
\]
then
\[
\int_{E_p}D^*\widehat\omega_\varepsilon
=
\int_{E_q}\widehat\omega_\varepsilon.
\]
Since all exceptional curves have the same $\widehat\omega_\varepsilon$-area, the two-forms have the same pairing with a generating set of $H_2(X_0^J,\mathbb R)$. Hence
\[
[D^*\widehat\omega_\varepsilon]
=
[\widehat\omega_\varepsilon],
\]
and uniqueness of the Ricci-flat representative in this class gives
\[
D^*g_\varepsilon=g_\varepsilon.
\]
Together with the $L_i$-invariance proved above, this shows that
$g_\varepsilon$ is $G$-invariant.

\noindent\textit{Passing to the quotient:} Since the action $G \curvearrowright X_0^J$ is free, it follows that $P:=X_0^J/G$ is a smooth manifold. Since $g_\varepsilon$ is $G$-invariant, it descends to a Ricci-flat metric on $P$, which we continue to denote by $g_\varepsilon$. The  $G$-equivariant resolution map $\pi_J$ descends to a $C\varepsilon^{1/2}-$Gromov--Hausdorff approximation  $$\pi_J:(P,g_\varepsilon) \longrightarrow (X,g_0),$$ and therefore $(P,g_\varepsilon)\xrightarrow[\varepsilon\to 0]{\mathrm{GH}}X$. Since $X_0^J$ is simply connected and the $G$-action is free, $\pi_1(P)\cong \mathbb Z_2\times \mathbb{Z}_2$. Finally, since $(X_0^J,g_\varepsilon)$ is a Ricci-flat K\"ahler $K3$ surface, it is hyperk\"ahler. Therefore the quotient
$(P,g_\varepsilon)$ is locally hyperk\"ahler.

\end{proof}
\subsection{The Sequence $(M,h_\varepsilon)$}
Here we construct the sequence of Ricci-flat metrics on the manifold $M$ appearing in \Cref{thm:main2} and verify the properties listed there. For this second sequence, we work with the intermediate orbifold
\begin{equation}
Z:=X_0/\langle L_i\rangle.
\label{eq}
\end{equation}
Recall that $L_i^2=-\mathrm{Id}$ on $\mathbb T^4$, while the induced map on $X_0=\mathbb T^4/{\pm\mathrm{Id}}$ is an involution. Hence, the four $L_i$-fixed Kummer points listed in \eqref{eq:fixpointsLi} give rise to four singularities of $Z$ with local group $\mathbb Z_4$, modeled on $\mathbb H/\langle L_i\rangle$. The remaining twelve Kummer points form six $\langle L_i\rangle$-orbits and give six singularities with local group $\mathbb Z_2$, modeled on $\mathbb H/\{\pm\mathrm{Id}\}$.

Since $D$ commutes with $L_i$ on $X_0$, it descends to an involution of $Z$, whose action is free. Indeed, if a point of $Z$ were fixed by the induced involution, then one of its lifts to $X_0$ would be fixed either by $D$ or by $L_iD$. Since
\[
\operatorname{Fix}(D)
=
\operatorname{Fix}(L_iD)
=
\varnothing
\]
by \eqref{eq:thm2-fixedsets}, no such point exists.

We shall choose a complex structure $\widetilde J$ on $Z$ for which the action of $D$ extends freely to the complex resolution $Z^{\widetilde J}$. The second manifold will then be obtained as
\[
M:=Z^{\widetilde J}/\langle D\rangle;
\]
see \Cref{fig:diagram2}.

\begin{figure}[htbp]
\centering
\begin{tikzpicture}[>=stealth]

\node (H)   at (0,3.8) {$\mathbb H/\Lambda$};

\node (X0J) at (-2.8,2) {$X_0^J$};
\node (X0)  at (0,2)    {$X_0$};

\node (Z)   at (0,0.4) {$Z$};
\node (ZJ)  at (3.5,0.4) {$Z^{\widetilde J}$};

\node (M1)  at (-2.8,-1.5) {$P$};
\node (X1)  at (0,-1.5) {$X=X_0/G$};
\node (Q) at (3.5,-1.5) {$Z^{\widetilde J}/\langle D\rangle=M$};

\node[anchor=east] at (X0J.west)
    {$G=\langle L_i,D\rangle\curvearrowright$};

\node[anchor=west] at (ZJ.east)
    {$\curvearrowleft\langle D\rangle$};

\draw[->]
    (H) -- node[right] {$/\{\pm\mathrm{Id}\}$} (X0);

\draw[->]
    (X0J) -- node[above] {$\pi_J$}
    node[below] {\scriptsize $G$-equivariant} (X0);

\draw[->]
    (X0J) -- node[left] {$/G\ \text{(free)}$} (M1);

\draw[->]
    (X0) -- node[right] {$/\langle L_i\rangle$} (Z);

\draw[->]
    (ZJ) -- node[above] {$\pi_{\widetilde J}$}
    node[below] {\scriptsize $\langle D\rangle$-equivariant} (Z);

\draw[->]
    (Z) -- node[right] {$/\langle D\rangle$} (X1);

\draw[->]
    (ZJ) -- node[right] {$/\langle D\rangle \ \text{(free)}$} (Q);

\draw[->]
    (M1) -- node[above] {$\simeq$}
    node[below] {\scriptsize $\varepsilon$-GH} (X1);

\draw[->]
    (Q) -- node[above] {$\simeq$}
    node[below] {\scriptsize $\varepsilon$-GH} (X1);
\end{tikzpicture}
\caption{Diagram of the construction.}
\label{fig:diagram2}
\end{figure}

\begin{lemma}
\label{lem:thm2-first-M}
There exist an orthogonal linear complex structure $\widetilde J$ on $Z$, a complex resolution
\[
\pi_{\widetilde J}:
Z^{\widetilde J}\longrightarrow Z,
\]
and constants $\varepsilon_0,C>0$ with the following properties.

\begin{enumerate}[topsep=0pt,label=\arabic*.]
\item The surface $Z^{\widetilde J}$ is a $K3$ surface and, in particular, is simply connected.

\item The residual involution $D$ on $Z$ is antiholomorphic with respect to $\widetilde J$ and extends to a free antiholomorphic involution
\[
D:
Z^{\widetilde J}\longrightarrow Z^{\widetilde J}.
\]
\item For every $0<\varepsilon<\varepsilon_0$, $Z^{\widetilde J}$ admits a $D$-invariant Ricci-flat K\"ahler metric $\widetilde h_\varepsilon$ such that the resolution map
\[
\pi_{\widetilde J}: (Z^{\widetilde J},\widetilde h_\varepsilon) \longrightarrow (Z,g_0)
\]
is a $D$-equivariant $C\varepsilon^{1/2}$-Gromov--Hausdorff approximation.
\end{enumerate}
Consequently,
\begin{equation}
M
:=
Z^{\widetilde J}/\langle D\rangle
\label{eq:M1-def}
\end{equation}
is a smooth manifold satisfying
\[
\pi_1(M)\cong\mathbb Z_2,
\]
and the descended Ricci-flat metrics $h_\varepsilon$ satisfy
\[
d_{GH}\bigl((M,h_\varepsilon),X\bigr) < C\varepsilon^{1/2}.
\]
\end{lemma}

\begin{proof}

Consider $\widetilde{J} = R_i$. In $R_i$-holomorphic coordinates, $q = z_1 + j z_2$, the action of $L_i$ reads as $(z_1,z_2) \mapsto (iz_1, -iz_2)$. In particular, $L_i$ is holomorphic and $L_i \in SU(2,\widetilde{J})$. Therefore, $Z$ has ADE singularities. More precisely, it has four $A_3$ singularities and six $A_1$ singularities. Let
\[
\pi_{\widetilde J}:Z^{\widetilde J}\longrightarrow Z
\]
be the minimal resolution.

This is precisely the orbifold considered in \Cref{thm:1-2}. Hence \Cref{thm:1-2} implies that $Z^{\widetilde J}$ is a $K3$ surface and provides Ricci-flat K\"ahler metrics for which $\pi_{\widetilde J}$ is a $C\varepsilon^{1/2}$-Gromov--Hausdorff approximation. It remains to extend the residual involution $D$ and to verify that the Ricci-flat metrics can be chosen to be $D$-invariant.

The linear part of $D$, as given in \eqref{eq:thm2-LiD}, satisfies
\[
(L_uR_{u^{-1}})\circ R_i
=
-R_i\circ(L_uR_{u^{-1}}).
\]
It follows that $D$ is antiholomorphic with respect to $\widetilde J$. As in the proof of \Cref{lemma:example1_interm}, it therefore extends uniquely through the minimal resolution to an antiholomorphic involution
\[
D:Z^{\widetilde J}\longrightarrow Z^{\widetilde J}.
\]
Moreover, since $D$ acts freely on $Z$, its extension acts freely on $Z^{\widetilde J}$ as well. The resolution map is $D$-equivariant.

\noindent\textit{Invariant Ricci-flat K\"ahler metrics:}
As in the previous arguments, by uniqueness of the Ricci-flat representative, it suffices to show that
\[
[-D^*\widehat{\omega}_\varepsilon] = [\widehat{\omega}_\varepsilon] \quad\text{in }H^2(Z^{\widetilde J},\mathbb R),
\]
where $\widehat{\omega}_\varepsilon$ is the approximating K\"ahler form used in the perturbation argument; compare with \eqref{eq:gluingdata}.

In this case, $H_2(Z^{\widetilde J},\mathbb R)$ has a basis consisting of eighteen exceptional classes, three for each of the four $A_3$ singularities and one for each of the six $A_1$ singularities, together with four classes inherited from the torus. The required identity on the toral classes follows as before: $D$ is an antiholomorphic isometry of the flat orbifold, and $\widehat{\omega}_\varepsilon$ agrees with the pullback of its K\"ahler form away from the gluing regions.

On the other hand, $D$ permutes the exceptional curves and reverses their complex orientations. Thus, it remains to arrange the exceptional curves exchanged by $D$ so that they have the same $\widehat{\omega}_\varepsilon$-area. These areas are determined by the local Ricci-flat K\"ahler forms $\eta_p$ used in the gluing; see \Cref{eq:etap,eq:gluingdata}.

Hattori's construction uses the forms $\eta_p$ for $p\in\mathcal S(Z)$ provided by Joyce's theorem; see \cite[Theorem~2.8]{Hattori2026Harmonic} and \cite[Theorem~3.3]{Joyce2001ALE}. On a fixed minimal resolution, this theorem provides a unique Ricci-flat ALE K\"ahler metric in each K\"ahler class admitting an ALE K\"ahler representative, together with the required asymptotic estimates for its K\"ahler potential. At each $A_3$ singularity, we choose a class for which the three exceptional curves have the same positive area; such a choice is available, for instance, from Kronheimer's parametrization of the $A_3$ ALE spaces \cite{Kronheimer1989ALE}.

We use the same normalized $A_3$ model at all four $A_3$ singularities, with the same gluing scale. Likewise, we use the same normalized Eguchi--Hanson model at all six $A_1$ singularities. With these choices, exceptional curves exchanged by $D$ have equal $\widehat{\omega}_\varepsilon$-area. Consequently, for every exceptional curve $E$, oriented by its complex structure,
\[
\int_E-D^*\widehat{\omega}_\varepsilon
=
\int_{D(E)}\widehat{\omega}_\varepsilon
=
\int_E\widehat{\omega}_\varepsilon,
\]
where $D(E)$ is also given its complex orientation. Together with the identity on the toral classes, this proves
\[
[-D^*\widehat{\omega}_\varepsilon]
=
[\widehat{\omega}_\varepsilon].
\]
Hence $-D^*$ preserves the K\"ahler class of $\widehat{\omega}_\varepsilon$. By uniqueness of the Ricci-flat K\"ahler metric in this class, the resulting Ricci-flat metric $\widetilde h_\varepsilon$ is $D$-invariant. Since $D$ acts freely, this metric descends to a Ricci-flat metric $h_\varepsilon$ on
\[
M
=
Z^{\widetilde J}/\langle D\rangle.
\]
Since $Z^{\widetilde J}$ is simply connected and $D$ acts freely,
\[
\pi_1(M)\cong\mathbb Z_2.
\]
Moreover, the $D$-equivariant resolution map descends to a $C\varepsilon^{1/2}$-Gromov--Hausdorff approximation
\[
(M,h_\varepsilon)
\longrightarrow
Z/\langle D\rangle
=
X.
\]
Thus
\[
(M,h_\varepsilon)
\xrightarrow[\varepsilon\to0]{\mathrm{GH}}
X.
\]
Finally, since $(Z^{\widetilde J},\widetilde h_\varepsilon)$ is a Ricci-flat K\"ahler $K3$ surface, it is hyperk\"ahler, and therefore its free isometric quotient $(M,h_\varepsilon)$ is locally hyperk\"ahler.

\end{proof}

We now combine the two constructions above to prove \Cref{thm:main2}.

\begin{proof}[Proof of \Cref{thm:main2}]
By \Cref{thm:thm2-first-P}, for every sufficiently small $\varepsilon>0$ there is a Ricci-flat metric $g_\varepsilon$ on
\[
P=X_0^J/G
\]
such that
\[
d_{GH}\bigl((P,g_\varepsilon),X\bigr) <C\varepsilon^{1/2}.
\]
Likewise, \Cref{lem:thm2-first-M} provides a Ricci-flat metric $h_\varepsilon$ on
\[
M=Z^{\widetilde J}/\langle D\rangle
\]
satisfying
\[
d_{GH}\bigl((M,h_\varepsilon),X\bigr)<C\varepsilon^{1/2}.
\]
After decreasing $\varepsilon_0$ and enlarging $C$ if necessary, the
same constants may be used for both families. Hence
\[
(P,g_\varepsilon)\xrightarrow{GH}X, \qquad (M,h_\varepsilon)\xrightarrow{GH}X \qquad\text{as }\varepsilon\to0.
\]
Moreover, the same two lemmas give
\[
\pi_1(P)\cong \mathbb Z_2\times\mathbb Z_2, \qquad \pi_1(M)\cong \mathbb Z_2.
\]
The uniform diameter and volume bounds follow exactly as in the proof of \Cref{thm:main1}: the diameter bounds follow from convergence to the compact orbifold $X$, while the Calabi--Yau normalization constants in the two desingularizations tend to $1$, giving a uniform positive lower bound for the volumes after passing to the finite quotients.

Finally, by \Cref{thm:thm2-first-P,lem:thm2-first-M}, the universal covers of $P$ and $M$ are respectively $X_0^J$ and $Z^{\widetilde J}$, both equipped with Ricci-flat K\"ahler metrics on $K3$ surfaces. Hence both universal covers are hyperk\"ahler, and therefore $(P,g_\varepsilon)$ and $(M,h_\varepsilon)$ are locally hyperk\"ahler.
\end{proof}

\section{Stability of the Positive Index}\label{sec:positive-index-stability}

In this section, we establish the stability of the positive index under non-collapsing Gromov--Hausdorff convergence when considering the orientation, and derive its consequences for the fundamental-group instability of locally hyperk\"ahler $4$-manifolds.

\subsection{Stability of the Positive Index for K\"ahler 4-Manifolds with Two-Sided Ricci Bounds}\label{subsec:StabPositiveIndexProof}

We begin by briefly explaining the context and introducing the notation used throughout the section.

\begin{definition}
\label{def:oriented-locally-kahler}
An oriented Riemannian $4$-manifold $(Y,g,o)$ is called \emph{oriented locally K\"ahler} if every point $p\in Y$ admits an open neighborhood $U$ and a complex structure $J_U$ on $U$ such that $(U,g|_U,J_U)$ is K\"ahler and the orientation induced by $J_U$ agrees with $o|_U$. The local complex structures $J_U$ are not required to agree on overlaps. In particular, we do not require $(Y,g)$ to admit a global K\"ahler structure inducing $o$.
\end{definition}
Throughout this section, we consider sequences of closed, connected, oriented locally K\"ahler $4$-manifolds $(Y_i,g_i,o_i)$ satisfying
\begin{equation}\label{eq:stability-geometric-bounds}
-\Lambda g_i\leq\operatorname{Ric}(g_i)\leq\Lambda g_i,
\qquad
\diam(Y_i,g_i)\leq D,
\qquad
\vol (Y_i,g_i)\geq \mathrm v>0, 
\end{equation}
and converging in the Gromov--Hausdorff sense to a compact metric space $(X,d_X)$.\footnote{By Gromov's compactness theorem, every sequence satisfying \eqref{eq:stability-geometric-bounds} admits a Gromov--Hausdorff convergent subsequence. Moreover, after passing to a further subsequence, the diffeomorphism type is fixed by \cite[Theorem 1.12]{CheegerNaber2015}. Consequently,  one may take sequences $(Y,g_i,o)$ without loss of generality.}

By \cite[Corollary 8.84]{CheegerNaber2015} and \cite{Anderson1990}, the limit $X$ is a compact Riemannian orbifold with finitely many isolated singular points
\[
\mathcal{S}(X)=\{p_1,\ldots,p_s\}.
\]
More precisely, $X$ is a smooth orbifold with finite local groups
$\Gamma_j\subset O(4)$ at the singular points $p_j$, and the limit metric $g_X$ extends continuously across $\mathcal{S}(X)$ as a positive-definite orbifold metric and is $C^{1,\beta}_{\mathrm{loc}}$ on $X_{\mathrm{reg}}:= X \backslash \mathcal S(X)$ for every $0<\beta<1$. Moreover, there are connected open sets $\Omega_i$ exhausting $X_{\mathrm{reg}}$ and smooth embeddings
\[
F_i:\Omega_i\longrightarrow Y_i
\]
such that
\begin{equation}
F_i^*g_i\longrightarrow g_X
\qquad\text{in }C^{1,\beta}_{\mathrm{loc}}(X_{\mathrm{reg}})
\end{equation}
for every $0<\beta<1$. 

The orientations $o_i$ pull back via the convergence maps $F_i$ to orientations $F_i^*o_i$ of the connected sets $\Omega_i$. Since $X_{\mathrm{reg}}$ admits an exhaustion by the connected orientable open sets $\Omega_i$, it is itself connected and orientable. Consequently, each pullback orientation $F_i^*o_i$ extends uniquely to an orientation of $X_{\mathrm{reg}}$. Moreover, the local orbifold groups preserve this orientation, so it extends uniquely to an orbifold orientation of $X$.
\begin{definition}\label{def:oriented-convergence}
Consider a sequence of oriented $4$-manifolds $(Y_i,g_i,o_i)$
satisfying \eqref{eq:stability-geometric-bounds} and converging
in the Gromov--Hausdorff sense to a Riemannian orbifold
$(X,g_X)$. Let $F_i:\Omega_i\to Y_i$ be convergence maps as above,
and let $o_X$ be an orbifold orientation of $X$. We say that the sequence $(Y_i,g_i,o_i)$ induces $o_X$ through the convergence maps $F_i$ if
\[
F_i^*o_i=o_X|_{\Omega_i}
\]
for all sufficiently large $i$.
\end{definition}
\begin{remark}
Since $X$ is orientable and $X_{\mathrm{reg}}$ is connected, $X$ has exactly two orientations. In particular, after passing to a subsequence, the sequence $(Y_i,g_i,o_i)$ induces an orientation of $X$ through the chosen convergence maps. For a general notion of orientation on Ricci limit spaces and its stability under noncollapsed convergence, see
\cite[Definition~1.1 and Theorem~1.3]{Honda2017}.
\end{remark}
\begin{remark}
As with the notion of convergence of functions on varying metric spaces, which is formulated relative to a choice of realization of the convergence, or equivalently to a choice of GH-approximations, the induced orientation may depend on the choice of the
convergence maps $F_i$ realizing the convergence, and on the chosen subsequence. 
\end{remark}

We next relate the intersection form of the limiting orbifold to that of a compact smooth manifold obtained by removing small neighborhoods of its singular points.

For each singular point \(p_j\in\mathcal S(X)\), choose an orbifold chart $(\widetilde U_j,\Gamma_j,\pi_j)$ centered at $p_j$, with pairwise disjoint images $U_j:=\pi_j(\widetilde U_j)\cong\widetilde U_j/\Gamma_j$. We take $\widetilde U_j=B_{\rho}(0)\subset\mathbb R^4$, $\pi_j(0)=p_j$, and $\Gamma_j\subset SO(4)$, where \(\Gamma_j\) acts freely on \(S^3\). For sufficiently small \(r>0\), set
\[
B_{j,r}:=\pi_j\bigl(\overline{B_r(0)}\bigr),
\qquad
C_r:=X\setminus\bigcup_{j=1}^s\operatorname{int}(B_{j,r}).
\]
Then \(C_r\) is a compact oriented smooth $4$-manifold with boundary
\begin{equation}
\partial C_r\cong\bigsqcup_{j=1}^s S^3/\Gamma_j.\label{eq:coreq}
\end{equation}
\begin{lemma}\label{lema:indexlemma}
Let $(X,o_X)$ be a compact oriented $4$-orbifold with isolated
singularities $p_1,\ldots,p_s$ and local groups
$\Gamma_j\subset SO(4)$ acting freely on $S^3$. For sufficiently
small $r>0$, let $C=C_r$ be the core defined above, endowed with
the orientation induced by $o_X$. Let $Q_X$ and $Q_C$ be the
intersection forms on $H^2(X;\mathbb R)$ and
$H^2(C,\partial C;\mathbb R)$, respectively. Set
$B:=\bigcup_{\ell=1}^s B_{\ell,r}$, and let
$j:(X,\varnothing)\hookrightarrow(X,B)$ be the natural inclusion.

Then $j^*:H^2(X,B;\mathbb R)\to H^2(X;\mathbb R)$ is an
isomorphism, and the map
\[
\Phi:=\operatorname{ex}^*\circ(j^*)^{-1}
\colon H^2(X;\mathbb R)\longrightarrow
H^2(C,\partial C;\mathbb R),
\]
is an isomorphism that pulls back $Q_C$ to $Q_X$, where
$\operatorname{ex}^*$ is the excision map. Moreover, both forms
are nondegenerate.
\end{lemma}
\begin{proof}
The result is standard and can be proved in several ways; we briefly
outline one. All cohomology groups below have real coefficients.
Since each component of $B$ is contractible, $H^1(B)=H^2(B)=0$.
The long exact sequence of $(X,B)$ therefore shows that $j^*$ is an
isomorphism. The excision map $\operatorname{ex}^*$ is an isomorphism
by excision and a collar deformation of $\partial C$.

A fundamental cycle for $X$ is obtained by coning off, inside $B$,
the boundary of a relative fundamental cycle for $C$. Hence
\[
j_*[X]=\operatorname{ex}_*[C,\partial C]\quad\text{in }H_4(X,B).
\]
Therefore, naturality of the cup product and the evaluation pairing gives the assertion about $Q_C$ and $Q_X$.

Now, note that $\partial C$ is a disjoint union of real homology $3$-spheres, and therefore using the long exact sequence of the pair $(C,\partial C)$ we get that the natural map $\lambda:H^2(C,\partial C)\to H^2(C)$, which forgets the relative condition, is an isomorphism. By Poincaré--Lefschetz duality, the pairing $P(a,u):=\langle a\smile u,[C,\partial C]\rangle$ on
$H^2(C)\times H^2(C,\partial C)$ is nondegenerate. Since
$Q_C(u,v)=P(\lambda u,v)$, this also proves the nondegeneracy of
$Q_C$ and hence of $Q_X$.
\end{proof}

We now prove \Cref{thm:positive-index-stability}.

\begin{proof}[Proof of \Cref{thm:positive-index-stability}]
Let $(Y_i,g_i,o_i)$ and $F_i:\Omega_i \subset X \to Y_i$ be as in the statement.

\noindent\emph{Lower semicontinuity of the positive index:} Fix a sufficiently small $r>0$ and set $C=C_r$. For all sufficiently large $i$, we have $C\subset\Omega_i$. Since $F_i$ restricts to an orientation-preserving diffeomorphism from $C$ onto $C_i:=F_i(C)$, \Cref{lema:indexlemma} gives
\begin{equation}\label{eq:b_2C_iandX}
b_2^+(C_i)=b_2^+(C)=b_2^+(X).
\end{equation}
Now set $p:=b_2^+(X)$. Using the identification
\[
H^2(C_i,\partial C_i;\mathbb R)
\cong H_c^2(\operatorname{int}C_i;\mathbb R),
\]
choose closed forms
$\alpha_1,\ldots,\alpha_p\in\Omega_c^2(\operatorname{int}C_i)$
whose classes span a positive-definite subspace for the intersection
form of $C_i$. Extending these forms by zero to $Y_i$ gives closed
forms with the same intersection matrix. Therefore,
\begin{equation}
b_2^+(Y_i)\geq b_2^+(C_i) \geq p=b_2^+(X) \label{eq:lowerboundstab}
\end{equation}
for all sufficiently large $i$.

\medskip
\noindent\emph{Self-dual harmonic forms do not concentrate on bubbles:} Let $\alpha_i$ be a self-dual harmonic $2$-form on $Y_i$. The Weitzenb\"ock formula for self-dual $2$-forms (see, for instance \cite[Equation~(15)]{LeBrun2003}; see also \cite{Bourguignon1981}) gives
\[
0=\Delta_H\alpha_i
=\nabla^*\nabla\alpha_i+Q_i\alpha_i,
\qquad
Q_i:=\frac{s_i}{3}\operatorname{Id}-2W_i^+.
\]
Let $U\subset Y_i$ be an open set on which $g_i$ is K\"ahler with respect to a complex structure $J_U$ inducing the orientation $o_i$, and let $\omega_U$ be the corresponding K\"ahler form. By \cite[Proposition~16.62]{Besse1987}, $W_i^+$ acts by $s_i/6$ on $\mathbb R\omega_U$ and by $-s_i/12$ on its orthogonal complement in $\Lambda^+|_U$. Consequently,
\[
Q_i|_U=\frac{s_i}{2}P_{i,U},
\]
where $P_{i,U}$ is the orthogonal projection onto $(\mathbb R\omega_U)^\perp\subset\Lambda^+|_U$. Although $P_{i,U}$ is defined using a local K\"ahler structure, the curvature endomorphism $Q_i$ is globally defined on $\Lambda^+(Y_i)$. Since the Ricci bounds imply $|s_i|\leq4\Lambda$, we obtain $\langle Q_i\xi,\xi\rangle
\geq-2\Lambda|\xi|^2$ for every self-dual two-form $\xi$. Integrating the Weitzenb\"ock identity over the closed manifold $Y_i$ therefore gives
\begin{equation}\label{eq:stability-gradient-bound}
\|\nabla\alpha_i\|_{L^2}^2
\leq2\Lambda\|\alpha_i\|_{L^2}^2.
\end{equation}
The geometric bounds \eqref{eq:stability-geometric-bounds} yield a uniform scalar Sobolev inequality \cite{Gallot1988Sobolev}. Applying it to $|\alpha_i|$ and using Kato's inequality, $\abs{\nabla\abs{\alpha_i}}\leq\abs{\nabla\alpha_i} $, gives 
\begin{equation}\label{eq:stability-L4-bound}
\|\alpha_i\|_{L^4}^2 \leq A\bigl(
\|\nabla\alpha_i\|_{L^2}^2+
\|\alpha_i\|_{L^2}^2
\bigr) \leq A'\|\alpha_i\|_{L^2}^2,
\end{equation}
where $A$ and $A'$ depend only on $\Lambda,D,\mathrm v$.

In particular, if $\|\alpha_i\|_{L^2}=1$, then H\"older's inequality together with Bishop--Gromov gives,
\begin{equation}\label{eq:stability-no-concentration}
\int_{B_r(x)}|\alpha_i|^2\,dV_{g_i}
\leq A''r^2,  
\end{equation}
uniformly in $i$ and $x$.

\medskip
\noindent\emph{Upper semicontinuity of the positive index:} Suppose, toward a contradiction, that $b_2^+(Y_i,o_i)>b_2^+(X)=:p$ for infinitely many $i$. Passing to this subsequence, choose an $L^2$-orthonormal family of self-dual harmonic
two-forms
\[
\alpha_i^1,\ldots,\alpha_i^{p+1}.
\]
By \eqref{eq:stability-gradient-bound}, the forms and their covariant derivatives are uniformly bounded in
$L^2$. Together with the $C^{1,\beta}_{\mathrm{loc}}$ convergence of $F_i^*g_i$, this implies that the pullbacks
$F_i^*\alpha_i^a$ are uniformly bounded in $W^{1,2}$ on every compact subset of $X_{\mathrm{reg}}$, with respect to $g_X$ or, equivalently, with respect to $F_i^*g_i$ for $i$ large enough.
Rellich compactness, followed by a diagonal subsequence argument over a compact exhaustion of $X_{\mathrm{reg}}$, gives two-forms $\alpha^1,\ldots,\alpha^{p+1}$ such that
\[
F_i^*\alpha_i^a\longrightarrow\alpha^a
\qquad\text{strongly in }
L^2_{\mathrm{loc}}(X_{\mathrm{reg}}).
\]
After passing to a further subsequence, the convergence also holds almost everywhere on  $X_{\mathrm{reg}}$. By Fatou's lemma and \eqref{eq:stability-L4-bound}, we obtain \( \|\alpha^a\|_{L^4(X_{\mathrm{reg}},g_X)}^2\leq A'.\)
Passing to the limit and using the fact that the original forms are closed and self-dual, we obtain
\begin{equation}\label{eq:stability-limit-equations}
d\alpha^a=0
\quad\text{in the distributional sense},
\qquad
*_{g_X}\alpha^a=\alpha^a
\quad\text{almost everywhere}.
\end{equation}
We next show that the limiting family is still $L^2$-orthonormal. Fix a sufficiently small $r>0$. As $C_r$ is compactly contained in $X_{\mathrm{reg}}$, it lies in $\Omega_i$ for all sufficiently large $i$. Since $F_i$ realize the Gromov--Hausdorff convergence, for all sufficiently large $i$ we have $Y_i\setminus F_i(C_r)\subset \bigcup_{j=1}^s B_{Cr}(p_{i,j})$, where $p_{i,j}\to p_j$. Hence, by Bishop--Gromov, $\operatorname{vol}_{g_i}(Y_i\setminus F_i(C_r))\leq Cr^4$, where $C$ depends on the lower Ricci curvature bound and the orbifold $X$. By \eqref{eq:stability-no-concentration},
\begin{equation}
\limsup_{i\to\infty}
\int_{Y_i\setminus F_i(C_r)}
|\alpha_i^a|^2\,dV_{g_i}
\leq Kr^2.\label{eq:stability-no-concentration-2}
\end{equation}

Since the original family is orthonormal,
\[
\delta_{ab}
=
\int_{F_i(C_r)}
\langle\alpha_i^a,\alpha_i^b\rangle\,dV_{g_i}
+
\int_{Y_i\setminus F_i(C_r)}
\langle\alpha_i^a,\alpha_i^b\rangle\,dV_{g_i}.
\]
Strong local $L^2$ convergence and convergence of the metrics imply
\[
\lim_{i\to\infty}
\int_{F_i(C_r)}
\langle\alpha_i^a,\alpha_i^b\rangle\,dV_{g_i}
=
\int_{C_r}
\langle\alpha^a,\alpha^b\rangle_{g_X}\,dV_{g_X}.
\]
Using \eqref{eq:stability-no-concentration-2} and Cauchy-Schwarz, it follows that
\[
\left|
\delta_{ab}
-
\int_{C_r}
\langle\alpha^a,\alpha^b\rangle_{g_X}\,dV_{g_X}
\right|
\leq Kr^2.
\]
Letting $r\to0$ gives
\begin{equation}\label{eq:stability-limit-orthonormality}
\int_{X_{\mathrm{reg}}}
\langle\alpha^a,\alpha^b\rangle_{g_X}\,dV_{g_X}
=\delta_{ab}.
\end{equation}
Since the limiting forms are self-dual,
\[
\int_{X_{\mathrm{reg}}}
\alpha^a\wedge\alpha^b
=
\int_{X_{\mathrm{reg}}}
\langle\alpha^a,\alpha^b\rangle_{g_X}\,dV_{g_X}
=\delta_{ab}.
\]

By \Cref{lem:closed-form-cutoff} below, there exist smooth closed two-forms $\theta^1,\ldots,\theta^{p+1}\in\Omega_c^2(\operatorname{int}C)$ satisfying
\[
\int_C\theta^a\wedge\theta^b=\int_{X_{\mathrm{reg}}}\alpha^a\wedge\alpha^b=\delta_{ab}.
\]
Consequently, the classes $[\theta^1],\ldots,[\theta^{p+1}]$ span a positive-definite subspace of dimension $p+1$ in $H^2_c(\operatorname{int}C;\mathbb R)$. This contradicts the equality $b_2^+(C)=b_2^+(X)=p$.

\end{proof}

\begin{remark}\label{rmk:stability}
The positive-index stability result was initially inspired by the work of Bando \cite{Bando1990} on degenerations of K\"ahler--Einstein 4-manifolds. In particular, Bando shows that suitable neighborhoods of the developing orbifold singularities are diffeomorphic to cyclic quotients of ALE gravitational instantons\cite[Theorem~3]{Bando1990}. In the hyperk\"ahler orientation, these instantons have vanishing positive index by Kronheimer's classification \cite{Kronheimer1989ALE}; the same holds for their orientation-preserving finite quotients. This suggests an alternative approach to the Kähler–Einstein case of the theorem, and it seems plausible that this approach could also be adapted to the more general setting of two-sided Ricci bounds considered here.

Our proof instead draws on ideas related to the spectral convergence of the Hodge Laplacian. In degree one, Honda and Mondino \cite{HondaMondino2025Poincare} combine the spectral convergence established in \cite{Honda2017Spectral} with topological arguments to obtain, as a byproduct, stability of the first Betti number in dimension four. This is a special case of the more general $b_1$-stability result discussed in the introduction.

In degree two, spectral convergence of the Hodge Laplacian can fail even for noncollapsed sequences of Ricci-flat K\"ahler $4$-manifolds; see \cite[Example~3.16]{Honda2017Spectral} and the references therein. Hattori \cite{Hattori2026Harmonic} studies such a degeneration in greater detail, showing that anti-self-dual harmonic $2$-forms can concentrate near the singularities and give rise, after rescaling, to harmonic forms on the bubbles. The key observation in our proof is that the locally K\"ahler condition, together with the geometric bounds, prevents this loss of $L^2$-mass through the bubbles for normalized self-dual harmonic two-forms.
\end{remark}

Some direct consequences of the previous theorem are the following.

\begin{corollary}\label{cor:positiveindex0bubbles}
Under the hypotheses of \Cref{thm:positive-index-stability}, for any fixed sufficiently small $r>0$, each connected component of
\[
Y_i\setminus F_i(\operatorname{int} C_r)
\]
has vanishing positive index for all sufficiently large $i$.
\end{corollary}
\begin{proof}
Fix $C=C_r$ and take $i$ sufficiently large so that $C\subset\Omega_i$. Set
\[
C_i:=F_i(C),
\qquad
V_i:=Y_i\setminus F_i(\operatorname{int} C).
\]
Then $C_i$ and $V_i$ are compact smooth manifolds with boundary, and
\[
Y_i=C_i\cup V_i,
\qquad
C_i\cap V_i=\partial C_i=\partial V_i.
\]
Every connected component of this common boundary is diffeomorphic to some $S^3/\Gamma_j$ and hence has vanishing first and second real cohomology. The Mayer--Vietoris sequence therefore gives
$$H^2(Y_i;\mathbb R)
\cong
H^2(C_i;\mathbb R)\oplus H^2(V_i;\mathbb R).$$
 Moreover, using the long exact sequences of pairs, the natural maps from relative to absolute second cohomology are isomorphisms on both pieces, because their boundaries are rational homology three-spheres. Thus classes on either piece can be represented by smooth closed forms compactly supported in its interior. Extending these representatives by zero to $Y_i$, classes coming from different pieces have zero intersection pairing. Consequently, the above Mayer--Vietoris isomorphism identifies the intersection form of $Y_i$ with the orthogonal direct sum of the intersection forms of $C_i$ and the connected components of $V_i$. In particular,
\[
b_2^+(Y_i,o_i)
=
b_2^+(C_i,o_i|_{C_i})
+
\sum_{U\in\pi_0(V_i)}
b_2^+(U,o_i|_U).
\]
By \eqref{eq:b_2C_iandX} and \Cref{thm:positive-index-stability},
\[
b_2^+(Y_i,o_i)
=
b_2^+(C_i,o_i|_{C_i}),
\]
where the induced orientation on $C_i$ is taken into account. Therefore,
\[
\sum_{U\in\pi_0(V_i)}
b_2^+(U,o_i|_U)=0.
\]
This proves the final assertion.
\end{proof}

\begin{corollary}\label{cor:stability2}
Let $(Y,g_i,o)$ be a sequence of closed, connected, oriented locally K\"ahler $4$-manifolds satisfying \eqref{eq:stability-geometric-bounds} and converging in the Gromov--Hausdorff sense to an orbifold $X$. Then either
\[
b_2^+(X,o_X)=b_2^-(X,o_X)=b_2^+(Y,o)
\]
for every orientation $o_X$ of $X$, or the sequence induces an orientation $o_X$ of $X$ through any choice of convergence maps $F_i:\Omega_i \subset X\longrightarrow Y
$. In this case,
\[
b_2^+(X,o_X)=b_2^+(Y,o).
\]
\end{corollary}
\begin{proof}
 If the sequence does not induce a fixed orientation on $X$, then there exist two subsequences inducing opposite orientations $o_X$ and $-o_X$ through the chosen convergence maps. Applying \Cref{thm:positive-index-stability} to each subsequence, we obtain
\(
b_2^+(Y,o)=b_2^+(X,o_X)
\)
and
\(
b_2^+(Y,o)=b_2^+(X,-o_X)=b_2^-(X,o_X).
\)
The last assertion follows immediately by contraposition.
\end{proof}

We conclude this section by stating and proving \Cref{lem:closed-form-cutoff}, which was used at the end of the proof of \Cref{thm:positive-index-stability}.

 \begin{lemma}\label{lem:closed-form-cutoff}
Let $\alpha^1,\ldots,\alpha^k\in L^4\Omega^2(X_{\mathrm{reg}},g_X)$ be weakly closed two-forms. For any sufficiently small $r>0$, there exist smooth closed two-forms $\theta^1,\ldots,\theta^k\in\Omega_c^2(\operatorname{int}C_r)$ such that
\[
\int_{C_r}\theta^a\wedge\theta^b=\int_{X_{\mathrm{reg}}}\alpha^a\wedge\alpha^b
\]
for every $a,b=1,\ldots,k$.
\end{lemma}

\begin{proof}
We first drop the superscripts. Fix a uniformizing chart and write $B=\widetilde U_j=B_R(0)$ for its Euclidean coordinate ball. Lift $\alpha$ to a $\Gamma_j$-invariant $L^4$ two-form $\widetilde\alpha$ on $B\setminus\{0\}$. Note that the lifted metric is uniformly equivalent to the Euclidean metric. Since $\widetilde\alpha\in L^2\Omega^2$, it is easy to check, by a standard cutoff argument, that $\int_B\widetilde\alpha\wedge d\varphi=0$ for every $\varphi\in \Omega_c^1(B)$, namely, $d\widetilde\alpha=0$ in the weak sense on $B$.

For a smooth two-form $\omega$ on $B_R\subset\mathbb R^4$, define the radial homotopy operator $T$ by
\[
(T\omega)(x)
:=
\int_0^1 t\,\iota_x\omega(tx)\,dt,
\]
where $\iota_x$ denotes contraction with the vector $x$. It is well known that if $\omega$ is closed, then $\omega=d(T\omega)$. Moreover, Minkowski's inequality and scaling give
\begin{equation}
\begin{aligned}
\|T\omega\|_{L^4(B_R)}
&\leq
R\int_0^1 t\,
\|\omega(t\,\cdot)\|_{L^4(B_R)}\,dt=
R\int_0^1
\|\omega\|_{L^4(B_{tR})}\,dt\leq
R\|\omega\|_{L^4(B_R)}.
\end{aligned}
\end{equation}
Fix $R'<R$. For sufficiently small $\delta>0$, mollifying the coefficients of $\widetilde\alpha$ with a radial mollifier gives smooth closed $\Gamma_j$-invariant forms $\widetilde\alpha_\delta$ on $B_{R'}$ such that $\widetilde\alpha_\delta\to\widetilde\alpha$ in $L^4(B_{R'})$. The estimate above implies that $T\widetilde\alpha_\delta\to T\widetilde\alpha=\widetilde\beta$ in $L^4(B_{R'})$. Since $d(T\widetilde\alpha_\delta)=\widetilde\alpha_\delta$, passing to the limit in the sense of distributions yields
\[
d\widetilde\beta=\widetilde\alpha.
\]
Finally, $T$ commutes with pullbacks by elements of $\Gamma_j$. Hence $\widetilde\beta$ is $\Gamma_j$-invariant and descends to an $L^4$ one-form on the quotient chart.

Apply this construction on each $B_{j,4r}$ to obtain a primitive $\beta_{j,r}$, taking $r$ sufficiently small that these balls lie inside their respective uniformizing charts. Choose a smooth radial cutoff $\chi_{j,r}$ equal to one on $B_{j,2r}$, supported in $\operatorname{int}B_{j,3r}$, and satisfying $|d\chi_{j,r}|\leq K/r$. Extending each $\chi_{j,r}\beta_{j,r}$ by zero outside its chart, on $X_{\mathrm{reg}}$ define
\[
\widehat\theta_r
:=
\alpha-\sum_jd(\chi_{j,r}\beta_{j,r}).
\]
This form is weakly closed, belongs to $L^2$, and vanishes on every punctured ball $B_{j,2r}$. Moreover, we claim that (restoring superscripts)
\[
\int_{X_{\mathrm{reg}}}\widehat\theta_r^a\wedge\widehat\theta_r^b=\int_{X_{\mathrm{reg}}}\alpha^a\wedge\alpha^b.
\]
It suffices to prove that, for each $j$, $\int_{B_{j,3r}}d(\chi\beta)\wedge\gamma=0$ for all $\gamma\in L^2\Omega^2(X_{\mathrm{reg}})$, with $d\gamma=0$ in the weak sense. Here, $\beta=\beta_{j,r}^a$ and $\chi=\chi_{j,r}$, as above, for some $a,j$. By the preceding construction, there exist smooth $\Gamma_j$-invariant one-forms $\beta_\delta$ on the uniformizing chart of $B_{j,7r/2}$ such that $\beta_\delta\xrightarrow[]{L^2}\beta$ and $d\beta_\delta\xrightarrow[]{L^2}d\beta$, where we identify $\beta$ with its lift to the uniformizing chart. It is easy to see that $d(\chi\beta_\delta)\xrightarrow[]{L^2}d(\chi\beta)$. By the same cutoff argument used above, the lift of $\gamma$ is weakly closed across the origin. Since $\chi\beta_\delta$ is smooth and compactly supported in the uniformizing chart, weak closedness gives
\[
\int_{B_{j,3r}}d(\chi\beta_\delta)\wedge\gamma=0,
\]
where we use the same notation for the descended forms. Taking $\delta\to0$ proves the desired identity.

Finally, $\widehat\theta_r$ is a weakly closed $L^2$ two-form with compact support in $M:=\operatorname{int}C_r$. For each degree $q$, let
\[
\mathcal D_2^q(M)
:=
\left\{
\omega\in L^2\Omega^q(M):
d\omega\in L^2\Omega^{q+1}(M)
\right\},
\]
where $d$ is understood in the weak sense. By \cite[Theorem~1]{GoldshteinKuzminovShvedov1984} (see also \cite[Theorem~3.14 and Remark~3.15]{Valette2021Poincare}), we may choose a smooth closed form $\theta_r\in\Omega_c^2(M)$ and a compactly supported $\eta_r\in\mathcal D_2^1(M)$ such that
\[
\theta_r-\widehat\theta_r=d\eta_r.
\]
Moreover, for any two such forms, we have
\[
\int_{C_r}\theta_r^a\wedge\theta_r^b
=
\int_{C_r}\widehat\theta_r^a\wedge\theta_r^b
=
\int_{C_r}\widehat\theta_r^a\wedge\widehat\theta_r^b.
\]
The first equality follows from weak integration by parts, since $\theta_r^b$ is smooth and $d\theta_r^b=0$. For the second, let $R_i:\mathcal D_2^q(M)\longrightarrow\Omega^q(M)\cap\mathcal D_2^q(M)$ be the regularization operators from the cited theorem. For all sufficiently large $i$, the form $\eta_{r,i}^b:=R_i\eta_r^b$ is smooth and compactly supported in $M$. Moreover,
\[
d\eta_{r,i}^b=R_i(d\eta_r^b)
\longrightarrow d\eta_r^b
\qquad\text{in }L^2.
\]
Thus, testing the weak equation $d\widehat\theta_r^a=0$ against $\eta_{r,i}^b$ and passing to the limit gives the second equality. Consequently,
\[
\int_{C_r}\theta_r^a\wedge\theta_r^b=\int_{X_{\mathrm{reg}}}\alpha^a\wedge\alpha^b,
\]
as desired.
\end{proof}

\subsection{Consequences for Ricci-Flat K\"ahler 4-Manifolds}

We now apply the theorem to the nonflat locally hyperk\"ahler Ricci-flat case.

\begin{corollary}[Oriented stability in the locally hyperk\"ahler case]
\label{cor:oriented-fundamental-group}
Consider two non-collapsing sequences of closed, connected, nonflat, locally hyperk\"ahler $4$-manifolds converging to the same orbifold $X$. Equip the manifolds with orientations compatible with their local hyperk\"ahler structures. If the two sequences induce the same orientation on $X$, then all sufficiently late terms of both sequences have the same diffeomorphism type.
In particular, their fundamental groups are eventually isomorphic.
\end{corollary}

\begin{proof}

By Hitchin's classification \cite[Theorem~1]{Hitchin1974}, such manifolds have a $K3$ manifold as their universal cover, and their fundamental groups are among
\[
1,\qquad \mathbb Z_2,\qquad \mathbb Z_2\times\mathbb Z_2.
\]
These three possibilities correspond, respectively, to $K3$ manifolds, Enriques manifolds, and  the quotient of an Enriques manifold by a free antiholomorphic involution (a Hitchin manifold). In the last case, the metric is only locally K\"ahler, but the orientation is globally defined, since antiholomorphic involutions in complex dimension two preserve orientation. Moreover, each of these three types has a unique underlying diffeomorphism type; see \cite{DegtyarevKharlamov1997}.

Writing $d=|\pi_1(Y)|$, multiplicativity of Euler characteristic and signature gives
\[
\chi(Y)=\frac{24}{d},
\qquad
\tau(Y)=-\frac{16}{d}.
\]
Since $b_1(Y)=b_3(Y)=0$,
\[
2b_2^+(Y)
=
\chi(Y)-2+\tau(Y)
=
\frac8d-2.
\]
Consequently,
\begin{equation}\label{eq:stability-group-index}
b_2^+(Y)=\frac4{|\pi_1(Y)|}-1,
\end{equation}
and the three groups correspond to the distinct values $3$, $1$, and $0$.

Since the two sequences induce the same orientation on $X$, \Cref{thm:positive-index-stability} implies that their positive indices eventually coincide. By \eqref{eq:stability-group-index}, their fundamental groups are eventually isomorphic, and hence all sufficiently late terms of both sequences are diffeomorphic.
\end{proof}

Finally, we prove   \Cref{cor:two-type-obstruction}.

\begin{proof}[Proof of \Cref{cor:two-type-obstruction}]
By \eqref{eq:stability-group-index}, the three fundamental-group types have positive indices $3$, $1$, and $0$, respectively. Fix an orientation $o_X$ of $X$. After passing to a subsequence inducing a fixed orientation on $X$, \Cref{thm:positive-index-stability} implies that the positive index of each fundamental-group type occurring along such a sequence belongs to
\[
\{b_2^+(X,o_X),b_2^-(X,o_X)\}.
\]
Since this set contains at most two distinct values, the three types cannot all occur over the same limit.

Moreover, if two distinct types converge to $X$, their positive indices differ, so by \Cref{thm:positive-index-stability} they must induce opposite orientations on $X$. Orienting $X$ by the type with larger positive index, the pair $(b_2^+(X),b_2^-(X))$ is therefore $(3,1)$ when the fundamental groups are $1$ and $\mathbb Z_2$, $(1,0)$ when they are $\mathbb Z_2$ and $\mathbb Z_2\times\mathbb Z_2$, and $(3,0)$ when they are $1$ and $\mathbb Z_2\times\mathbb Z_2$.

Finally, the approximating manifolds are locally hyperk\"ahler and hence Ricci-flat with $W^+=0$ in their respective orientations. Since the convergence is smooth on $X_{\mathrm{reg}}$ and the two sequences induce opposite orientations on $X$, we conclude that $W^+=W^-=0$ on $X_{\mathrm{reg}}$. The same smooth convergence also gives $\operatorname{Ric}=0$ on $X_{\mathrm{reg}}$. The curvature decomposition therefore gives $\operatorname{Rm}=0$ on $X_{\mathrm{reg}}$, and hence on $X$.

\end{proof}

\printbibliography

\end{document}